\documentclass[12pt]{article}
\usepackage{amsthm}
\usepackage{amsfonts}
\usepackage{amsmath}
\usepackage{amssymb}
\usepackage{colonequals}
\usepackage{hyperref}
\usepackage{tikz,geometry}
\usetikzlibrary{shapes,calc}
\usepackage[utf8]{inputenc}
\usepackage{etoolbox}
\newcommand{\QQ}{\mathbf{Q}}
\newtheorem{conjecture}{Conjecture}
\newtheorem{theorem}{Theorem}
\newtheorem{corollary}[theorem]{Corollary}
\newtheorem{lemma}[theorem]{Lemma}

\theoremstyle{definition}
\newtheorem{problem}[theorem]{Problem}

\newcommand{\abs}[1]{\left\lvert#1\right\rvert}
\newcounter{counterclaim}[section]

\newtheoremstyle{claims}{\topsep}{\topsep}{}{}{\bfseries}{}{.5em}{}
\theoremstyle{claims}
\newtheorem{megaclaim}[counterclaim]{}
\newcommand{\claim}[2]{\begin{megaclaim}\label{#1} #2 \end{megaclaim}}
\newcommand{\refclaim}[1]{\ref{#1}}

\newenvironment{tikzgraph}
  {\begin{tikzpicture}
      [vertex/.style={circle, draw=black, fill=white, inner sep=0.5pt, minimum
        size=7.5pt},
       n/.style={circle, draw=black, fill=white, inner sep=0.5pt, minimum
        size=7.5pt},
       p/.style={circle, draw=black, fill=black, inner sep=0.5pt, minimum
        size=7pt},%
       t/.style={regular polygon,regular polygon sides=4, draw=black, fill=white, inner sep=0mm, minimum
        size=8pt},%
       s/.style={regular polygon,regular polygon sides=4, draw=black, fill=white, inner sep=0mm, minimum
        size=11pt},%
       edge/.style={semithick},%
       font=\scriptsize
      ]\begin{scope}}
  {\end{scope}\end{tikzpicture}}
\newdimen\R
\newdimen\RR
\title{Do triangle-free planar graphs have exponentially many $3$-colorings?\thanks{This work was done within the scope of the International Associated Laboratory STRUCO.}}
\author{Zdeněk Dvořák\thanks{Computer Science Institute, Charles University, Prague, Czech Republic.
Supported by the Center of Excellence -- Inst. for Theor. Comp. Sci., Prague, project P202/12/G061 of
Czech Science Foundation and by the project LL1201 (Complex Structures: Regularities in Combinatorics and
Discrete Mathematics) of the Ministry of Education of Czech Republic.
E-mail: \texttt{rakdver@iuuk.mff.cuni.cz}.
}\and
Jean-Sébastien Sereni\thanks{Centre National de la Recherche Scientifique, (LORIA), Vandœuvre-lès-Nancy, France. E-mail: \texttt{sereni@kam.mff.cuni.cz}.}}
\date{}

\begin{document}

\maketitle

\begin{abstract}
      Thomassen conjectured that triangle-free planar graphs have an exponential number
      of $3$-colorings. We show this conjecture to be equivalent to the following statement:
      there exists a positive real~$\alpha$ such that whenever~$G$ is a planar graph
      and~$A$ is a subset of its edges whose deletion makes~$G$ triangle-free, there
      exists a subset~$A'$ of~$A$ of size at least~$\alpha|A|$ such that $G-(A\setminus A')$
      is $3$-colorable. This equivalence allows us to study restricted situations, where we can
      prove the statement to be true.
\end{abstract}

\section{Introduction}\label{sec:intro} 
A now classical theorem of Grötzsch~\cite{grotzsch1959} asserts that every triangle-free planar graph
is $3$-colorable. This statement spurred a lot of interest and,
over the years, many ingenious proofs have been found~\cite{DvoKawTho,thom-torus,ThoShortlist}.
The new proofs are simpler than the original argument, and often target further developments --- algorithmic aspects or extension
to other surfaces. In particular, refining some of his arguments, Thomassen~\cite{thom-many} established that every
planar graph of girth at least five has exponentially many --- in terms of the number of vertices ---
list colorings provided all lists have size at least three. This statement cannot be extended to planar graphs of girth at least four,
that is, triangle-free planar graphs, as Voigt~\cite{voigt1995} exhibited
a triangle-free planar graph~$G$ along with an assignment~$L$ of lists of size
three to the vertices of~$G$ such that $G$ is not $L$-colorable.
However, it could still be true that triangle-free planar graphs admit
exponentially many $3$-colorings. This was actually conjectured in~2007 by Thomassen~\cite[Conjecture~2.1(b)]{thom-many}.
The formulation we give implicitly uses a theorem by Jensen and Thomassen~\cite[Theorem~10]{JT00} that
the $3$-color matrix of a planar graph has full rank if and only if the graph has no triangle.
\begin{conjecture}\label{conj:thom-many}
     There exists a positive real number~$\beta$ such that every triangle-free planar
     graph~$G$ has at least~$2^{\beta\abs{V(G)}}$ different $3$-colorings.
\end{conjecture}

As reported earlier, Thomassen~\cite{thom-many}
proved the statement under the additional assumption that~$G$ has no $4$-cycle. In addition,
he proved that every triangle-free planar graph~$G$ admits at least $2^{\abs{V(G)}^{1/12}/20\,000}$ different
$3$-colorings. This lower bound, which is sub-exponential, was later improved by Asadi, Dvořák, Postle and Thomas~\cite{submany}
to~$2^{\sqrt{\abs{V(G)}/212}}$. In addition, 
Dvořák and Lidick\'y~\cite[Corollary~1.3]{cylgen-part2} proved the existence of an integer~$D$ such that
every triangle-free planar graph~$G$ with maximum degree at most~$\Delta$ has at least~$3^{\abs{V(G)}/\Delta^D}$ different $3$-colorings,
thereby confirming the analogue of Conjecture~\ref{conj:thom-many} for all classes of triangle-free planar graphs with bounded maximum degree.
Actually, this statement follows from another result of theirs~\cite[Corollary~1.2]{cylgen-part2}, which states the existence of an integer~$D$
such that if $G$ is a triangle-free planar graph
and $V'\subset V(G)$ is a subset of vertices of~$G$ such that every two distinct vertices in~$V'$ are at distance at least~$D$ in~$G$,
then any $3$-precoloring of the vertices in~$V'$ extends to a $3$-coloring of the whole graph~$G$. As we will see later on, precoloring
extension might be a useful tool to study the number of $3$-colorings of triangle-free planar graphs.

Summing-up, we see that Conjecture~\ref{conj:thom-many} is still widely open. Our goal is to show the equivalence between
Conjecture~\ref{conj:thom-many} and another statement dealing with a variation---a very natural one, in our opinion---of
the usual notion of coloring, which we now introduce.

For a function~$w\colon X\to \QQ^+$ and a set~$X'\subseteq X$,
let~$w(X')=\sum_{x\in X'} w(x)$.  A \emph{request graph} $(G,R_=,R_{\neq},w)$
consists of a graph~$G$, disjoint sets~$R_=$ and~$R_{\neq}$ of vertices of~$G$ of
degree two such that $R_=\cup R_{\neq}$ is an independent set in~$G$, and
a function~$w\colon R_=\cup R_{\neq}\to \QQ^+$.
The vertices in~$R_=\cup R_{\neq}$
are referred to as the \emph{requests} or \emph{request vertices}.
Let~$\varphi$ be a proper coloring of~$G$.  We say that a vertex~$r\in R_=$ is \emph{satisfied} if both its
neighbors have the same color, and a vertex~$r\in R_{\neq}$ is \emph{satisfied} if
its neighbors have different colors.  For~$\alpha>0$, we say that
a $3$-coloring~$\varphi$ \emph{satisfies $\alpha$-fraction of the requests} if,
letting~$R'$ be the set of satisfied vertices in~$R_=\cup R_{\neq}$, we have
$w(R')\ge\alpha w(R_=\cup R_{\neq})$.  
The following problem arises from the work of Asadi~\emph{et al.}~\cite{submany}.
\begin{problem}\label{prob:adpt13}
      Is there a positive real number~$\alpha$ such that every
      planar triangle-free request graph admits
      a $3$-coloring satisfying $\alpha$-fraction of its requests?
\end{problem}

\noindent
As it turns out, Problem~\ref{prob:adpt13} admits a positive answer if and only if
Conjecture~\ref{conj:thom-many} is true.
\begin{theorem}\label{thm-eq-pr-exp}
The following assertions are equivalent.
\begin{itemize}
\item[\textrm{(RGEN)}] There exists a positive real number~$\alpha$ such that every
      planar triangle-free request graph admits a $3$-coloring that satisfies
      $\alpha$-fraction of its requests.
\item[\textrm{(EXP)}] There exists a positive real number~$\beta$ such that every
      planar triangle-free graph~$G$ has at least~$2^{\beta\abs{V(G)}}$
      $3$-colorings.
\end{itemize}
\end{theorem}

Before going any further, we pause to clarify the relation between~(RGEN) and
the statement given in the abstract of this article, namely:

\medskip

\begin{itemize}
      \item[(TRIA)] \emph{There is a positive real number~$\alpha$ such that for every planar
graph~$G$ and every subset~$X$ of edges such that $G-X$ is triangle-free,
there exists a $3$-coloring~$c$ of~$G-X$ such that at least~$\alpha|X|$
edges in~$X$ join vertices of different colors under~$c$.}
\end{itemize}


\noindent
It suffices to subdivide each edge in~$X$ by a vertex placed in~$R_{\neq}$
to see that (TRIA) is implied by~(RGEN).  We thus realize that (TRIA) is equivalent to
the special case of~(RGEN) where~$R_=$ is empty and $w$ assigns each vertex in~$R_{\neq}$ weight~$1$.
As we see below in Theorem~\ref{thm-eq-neq}, this special case is in fact equivalent
to~(RGEN), establishing the equivalence between~(RGEN) and~(TRIA).

Theorem~\ref{thm-eq-pr-exp} is proved in Section~\ref{sec:eq-exp}.
Request graphs allow for different ways to address Conjecture~\ref{conj:thom-many},
making it possible to focus on finding just one coloring subject to given constraints
rather than many.  It is unclear whether this will turn out to be advantageous, as Problem~\ref{prob:adpt13}
appears to be quite difficult.  For example, in Section~\ref{sec:atvert}, we consider the special case
under the additional assumption that there are only non-equality requests and all the requests are incident with the same vertex
(that is, $R_==\varnothing$ and all the vertices in~$R_{\neq}$ have a common neighbor).
We manage to establish the following.
\begin{corollary}\label{cor-single} Let~$\alpha_0=1/5058$.  Consider a request
      graph~$(G,\varnothing,R_{\neq},w)$, where $G$ is planar triangle-free.  If all
      vertices of~$R_{\neq}$ have a common neighbor, then there exists
      a $3$-coloring of~$G$ satisfying $\alpha_0$-fraction of the requests.
\end{corollary}

As strong as the hypothesis of Corollary~\ref{cor-single} are,
the argument turns out to be unexpectedly involved.  Let~$v$ be a common neighbor
to all requests in~$R_{\neq}$, let~$T$ be the set of
vertices other than~$v$ adjacent to the requests in~$R_{\neq}$, and let~$S$ be the set of
non-request neighbors of~$v$.  Without loss of generality, we can give~$v$ color~$3$,
and thus we seek a coloring of the graph~$G'=G-v-R_{\neq}$ in which all vertices
in~$S$ and a constant fraction of the vertices in~$T$ only use colors from the list~$\{1,2\}$.

Since the vertices in~$S\cup T$ are incident with the same face of~$G'$, this is reminiscent
of a well-known result of Thomassen~\cite{thomassen1995-34} (Theorem~\ref{thm-3choos} below),
which implies that such a coloring exists whenever~$G'$ has girth at least~$5$ and~$S\cup T$ is an independent set.
As it turns out, the graph~$G'$ actually can have $4$-cycles, but these
are relatively easy to deal with (we can eliminate separating $4$-cycles \emph{via} a precoloring
extension argument, and $4$-faces can be reduced in a standard way by collapsing).
Nevertheless, while the set~$S$ is independent since~$G$ is triangle-free, the vertices in~$T$ can be adjacent
to other vertices in~$S\cup T$.

Suppose for a moment that the outer face of~$G'$ is bounded by an induced cycle~$C$.
By ignoring a constant fraction of the requests, we can assume that the distance in~$C$
between any two distinct vertices in $T$ is at least three.  Consequently, $G'[S\cup T]$ does not
contain a path on four vertices; it still can, however, contain $3$-vertex paths with endvertices
in~$S$ and the middle vertex in~$T$.  It would be convenient to have a variation of Thomassen's result
that allows such $3$-vertex paths with lists of size~$2$; but no such variation is known or even likely to hold.
Even a quite involved result of Dvořák and Kawarabayashi~\cite{dk} for $3$-list-coloring only allows $2$-vertex
paths with lists of size two (and even that only subject to the additional restriction that the distance between such
paths is at least three).  Overcoming these issues requires a combination of several partial coloring arguments together with
elimination of a part of interfering constraints in~$T$ using a result of Naserasr~\cite{homclebsch} on
odd distance coloring of planar graphs.

The paper is structured as follows. In Section~\ref{sec:ground}, we perform some ground work on~Problem~\ref{prob:adpt13} where
we show that it actually suffices to restrict the attention to request graphs with only non-equality (or only equality)
requests, and to unit weights; that is, it is sufficient to consider request graphs of the
form~$(G,R_=,\varnothing,\text{unit})$ or, equivalently, of the form~$(G,\varnothing,R_{\neq},\text{unit})$,
where~$\text{unit}\colon R_=\cup R_{\neq}\to \QQ^+$ is the function constantly equal to~$1$.
Section~\ref{sec:eq-exp} is devoted to proving our main result, Theorem~\ref{thm-eq-pr-exp}.
After introducing and strengthening some auxiliary results on list colorings in Section~\ref{sec-aux},
we prove Corollary~\ref{cor-single} in Section~\ref{sec:atvert}.


\section{Ground work on~Problem~\ref{prob:adpt13}}\label{sec:ground} 

We start by proving the following equivalences.
\begin{theorem}\label{thm-eq-eq}
Let~$\alpha$ be a positive real number.  The following assertions are equivalent.
\begin{itemize}
  \item[\textrm{(RGEN)}] Every planar triangle-free request graph has
        a $3$-coloring that satisfies $\alpha$-fraction of its requests.
  \item[\textrm{(RE)}] Every planar triangle-free request
        graph~$(G,R_=,\varnothing,w)$ has a $3$-coloring that satisfies
        $\alpha$-fraction of its requests.
  \item[\textrm{(REU)}] Every planar triangle-free request
        graph~$(G,R_=,\varnothing,\text{unit})$ admits a $3$-coloring that satisfies
        $\alpha$-fraction of its requests.
\end{itemize}
\end{theorem}

\begin{proof}
The implications $\text{(RGEN)}\Rightarrow\text{(RE)}\Rightarrow\text{(REU)}$ are trivial.

Suppose that~(REU) holds, and let~$(G,R_=,\varnothing,w)$ be a planar triangle-free
request graph.  Without loss of generality, we can multiply all the values of~$w$
by some integer, so that the values of~$w$ become integral.  Let~$G'$ be the graph
obtained from~$G$ by replacing each vertex~$r\in R_=$ by $w(r)$ clones, and
let~$R'_=$ be the set of all such clones.  By (REU) applied
to~$(G',R'_=,\varnothing,\text{unit})$, there exists a $3$-coloring of~$G'$
satisfying $\alpha$-fraction of its requests, and its restriction to~$G$ satisfies
$\alpha$-fraction of requests of~$(G,R_=,\varnothing,w)$.  Hence, (REU)
implies~(RE).

Suppose that~(RE) holds, and let~$(G,R_=,R_{\neq},w)$ be a request graph.
Let~$G'$ be the graph obtained from~$G$ by replacing each vertex of~$R_{\neq}$ as
depicted in Figure~\ref{fig-gadgets}(a). Let~$R'_=$ be the set of created vertices
that are depicted in the figure by a square containing~``$=$''.  Let~$w'$ be the
function matching $w$ on $R_=$ and giving each vertex of~$R'_=$ the weight of the
vertex of~$R_{\neq}$ it replaces.  Then $(G',R_=\cup R'_=,\varnothing,w')$ is
a planar triangle-free request graph, and any $3$-coloring of~$G'$ corresponds to
a $3$-coloring of~$G$ satisfying the same fraction of the requests.  Hence, (RE)
implies~(RGEN).
\end{proof}

\begin{figure}[!t]
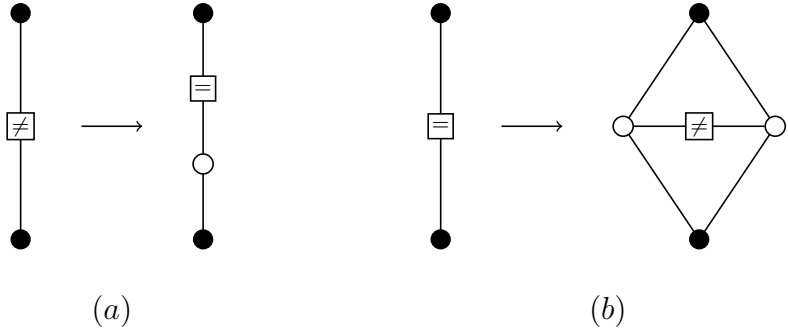

      \begin{center}
      \begin{tikzgraph}
            \draw[edge] (0,2) node[p]{} -- (0,.5) node[s]{$\neq$} -- (0,-1) node[p]{};
            \draw[edge,->] (.8,.5)--(1.6,.5);
            \draw[edge] (2.4,2) node[p]{} -- (2.4,1) node[s]{$=$} -- (2.4,0) node[n]{} -- (2.4,-1) node[p]{};
\draw (1.2,-1.2) node[below=\R/2]{{\normalsize $(a)$}};
\end{tikzgraph}
\hspace{25mm}
      \begin{tikzgraph}
            \draw[edge] (0,2) node[p]{} -- (0,.5) node[s]{$=$} -- (0,-1) node[p]{};
            \draw[edge,->] (.8,.5)--(1.6,.5);
            \draw[edge] (3.4,2) node[p]{} -- (2.4,.5) node[n](l){} --  (3.4,-1) node[p]{} -- (4.4,.5) node[n](r){} -- cycle;
            \draw[edge] (l)--(3.4,.5) node[s]{$\neq$} -- (r);
\draw (2.2,-1.2) node[below=\R/2]{{\normalsize$(b)$}};
\end{tikzgraph}
\end{center}
\caption{Gadgets showing equivalence of equality and inequality requests.}\label{fig-gadgets}
\end{figure}

Analogously (using the replacement from Figure~\ref{fig-gadgets}(b)) we obtain the following.
\begin{theorem}\label{thm-eq-neq}
Let~$\alpha$ be a positive real number.  The following assertions are equivalent.
\begin{itemize}
\item[\textrm{(RGEN)}] Every planar triangle-free request graph admits
      a $3$-coloring that satisfies $\alpha$-fraction of its requests.
\item[\textrm{(RN)}] Every planar triangle-free request
      graph~$(G,\varnothing,R_{\neq},w)$ has a $3$-coloring that satisfies
      $\alpha$-fraction of its requests.
\item[\textrm{(RNU)}] Every planar triangle-free request
      graph~$(G,\varnothing,R_{\neq},\text{unit})$ admits a $3$-coloring that
      satisfies $\alpha$-fraction of its requests.
\end{itemize}
\end{theorem}
Let us note that (RNU) is just a reformulation of the statement from the abstract, discussed as (TRIA)
earlier.

\section{Satisfying requests is equivalent to having exponentially many \texorpdfstring{$3$}{3}-colorings}\label{sec:eq-exp}
Theorem~\ref{thm-eq-eq} implies that we can establish Theorem~\ref{thm-eq-pr-exp} by proving the following statement.
\begin{theorem}\label{thm-eq-exp}
The following assertions are equivalent.
\begin{itemize}
\item[\textrm{(REU)}] There exists a positive real number~$\alpha$ such that every
      planar triangle-free request graph~$(G,R_=,\varnothing,\text{unit})$ has
      a $3$-coloring satisfying $\alpha$-fraction of its requests.
\item[\textrm{(EXP)}] There exists a positive real number~$\beta$ such that every
      planar triangle-free graph~$G$ has at least~$2^{\beta\abs{V(G)}}$
      $3$-colorings.
\end{itemize}
\end{theorem}

Showing~$\text{(EXP)}\Rightarrow\text{(REU)}$ is quite easy---we replace each request
by a large number of vertices of degree two with the same neighbors, and observe that
these vertices of degree two can only be colored in many ways if the neighbors are assigned
the same color, i.e., the request is satisfied. Thus, if the graph after the replacement
has exponentially many $3$-colorings, then a constant fraction of the requests must be satisfied.
The other implication~$\text{(REU)}\Rightarrow\text{(EXP)}$ is more involved and it
uses a number of auxiliary statements devised in order to prove the sub-exponential bounds
of Thomassen~\cite[Theorem~5.8]{thom-many} and Asadi~\emph{et al.}~\cite[Theorem~1.3]{submany}.
Essentially, the idea is to be able to place requests such that a $3$-coloring
satisfying a linear proportion of them will ensure properties that produce
many different $3$-colorings of the original graph. Mainly, we want the $3$-coloring
to produce $4$-faces the vertices of which avoid one of the three colors. We shall
thus pinpoint forced configurations of a minimal counter-example that allow us
to put requests which, if satisfied, produce such faces. We also need to prove
that there will be many such configurations, which is done using a decomposition
of the graph based on its separating $5$-cycles, as in the previous works on the topic.

We start by explaining why having many $4$-faces as mentioned above helps us,
through the following strengthening of a result of Thomassen~\cite{thom-many}.  For
a $3$-coloring of a plane graph, a face~$f$ is \emph{bichromatic} if the set of colors
assigned to the vertices incident to~$f$ has size two.

\begin{lemma}\label{lemma-manycolor}
Let~$G$ be a connected plane triangle-free graph with~$n\ge 3$ vertices, and for~$i\ge 4$, let~$s_i$ be the number of faces
of~$G$ of length exactly~$i$.  Let $\varphi$ be a $3$-coloring of~$G$, and let~$q$ be the number of bichromatic $4$-faces of~$G$.
Then $G$ has at least~$2^{(s^++8+q)/6}$ distinct $3$-colorings, where $s^+=s_5+2s_6+\ldots=\sum_{i\ge 5} (i-4)s_i$.
\end{lemma}
\begin{proof}
Let~$e$ be the number of edges of~$G$ and~$s$ the number of faces of~$G$.  By Euler's formula, $e+2=n+s$.
Furthermore, $2e=4s+s^+$, and thus $e=2n-4-s^+/2$.

For~$a,b\in\{1,2,3\}$ with~$a<b$, we define~$V_{ab}$ to be the set of vertices
of~$G$ colored by~$a$ or by~$b$, and we let $Q_{ab}$ be the set of~$4$-faces
of~$G$ with all incident vertices in~$V_{ab}$.  Let~$X_{ab}$ be a minimal set of
edges such that each face of~$Q_{ab}$ is incident with an edge of~$X_{ab}$.  By
the minimality of~$X_{ab}$, for every $e\in X_{ab}$ there exists a bichromatic
$4$-face~$f$ such that $e$ is the only edge of~$X_{ab}$ incident with~$f$, and
thus $G[V_{ab}]-X_{ab}$ has the same components as $G[V_{ab}]$.  Furthermore, $e$
may only be incident with two $4$-faces of~$Q_{ab}$, and thus $\abs{X_{ab}}\ge
\abs{Q_{ab}}/2$.  Let~$c_{ab}$ be the number of components of~$G[V_{ab}]$, set
$n_{ab}=\abs{V_{ab}}$ and $e_{ab}=\abs{E(G[V_{ab}])}$. Then
$e_{ab}-\abs{X_{ab}}\ge n_{ab}-c_{ab}$, and thus $e_{ab}\ge
n_{ab}-c_{ab}+\abs{X_{ab}}\ge n_{ab}-c_{ab}+\abs{Q_{ab}}/2$.

Summing these inequalities over all pairs of colors, we obtain
\[2n-4-s^+/2 =e=e_{12}+e_{23}+e_{13}\ge 2n-(c_{12}+c_{23}+c_{13})+q/2,\]
and thus
\[c_{12}+c_{23}+c_{13}\ge s^+/2+4+q/2.\]
By symmetry, we can assume that $c_{12}\ge c_{23}\ge c_{13}$, and thus
\[c_{12}\ge (s^++8+q)/6.\]
We can independently interchange the colors~$1$ and~$2$ on each component
of~$G[V_{12}]$, thereby obtaining $2^{c_{12}}$ different colorings of~$G$.
The statement of the lemma follows.
\end{proof}

\noindent
We also use the following result from Thomassen's paper.
\begin{lemma}[{Thomassen~\cite[Theorem~5.1]{thom-many}}]\label{lemma-2cols}
Let~$G$ be a plane triangle-free graph with outer face bounded by a cycle~$C$ of
length at most~$5$, and let~$\psi$ be a $3$-coloring of~$C$.  If $G\neq C$
and~$\psi$ does not extend to at least two $3$-colorings of~$G$, then there exists
a vertex~$v\in V(G)\setminus V(C)$ adjacent to two vertices of~$C$ of distinct
colors.
\end{lemma}

We need the following observation, which implicitly appears in the paper of Asadi~\emph{et al.}~\cite{submany}.
\begin{lemma}\label{lemma-minc}
Let~$\beta$ be a positive real number and let~$n$ be an integer such that every planar
triangle-free graph~$H$ with less than~$n$ vertices has at least~$2^{\beta\abs{V(H)}}$
distinct $3$-colorings.  Let~$d_0=\lfloor 1/\beta\rfloor$.  Let~$G$ be a planar
triangle-free graph with~$n$ vertices.  If $G$ has less than~$2^{\beta n}$
distinct $3$-colorings, then every vertex of~$G$ of degree at most~$d_0$ is
contained in a $5$-cycle.
\end{lemma}
\begin{proof}
      We prove the contrapositive. Assume that the graph~$G$ contains a vertex~$v$
      that has degree at most~$d_0$ and is not contained in any $5$-cycle. Let~$H$
      be the graph obtained from~$G-v$ by identifying all the neighbors of~$v$ to
      a single vertex.  Note that $H$ is planar and triangle-free, and every
      $3$-coloring of~$H$ extends to two distinct $3$-colorings of~$G$, as we can
      freely choose two different colors for~$v$.  By assumptions, we know that $H$ has
      at least $2^{\beta\abs{V(H)}}$ distinct $3$-colorings; hence
      $G$ has at~least $2^{\beta\abs{V(H)}+1}$ distinct $3$-colorings.
      Since $\abs{V(H)}\ge n-d_0\ge n-1/\beta$, we deduce that $\beta\abs{V(H)}+1\ge\beta
      n$, which concludes the proof.
\end{proof}

A \emph{$5$-cycle decomposition} of a plane graph~$G$ is a pair~$(T,\Lambda)$, where~$T$ is a rooted tree and~$\Lambda$ is a function
mapping each vertex of~$T$ to a subset of the plane, such that the following conditions hold.
\begin{itemize}
\item Let~$v$ be a vertex of~$T$.  If $v$ is the root of~$T$, then $\Lambda(v)$ is the whole plane, and otherwise $\Lambda(v)$
is the open disk bounded by a separating $5$-cycle of~$G$.
\item Let~$u$ and $v$ be distinct vertices of $T$.  If $u$ is a descendant of~$v$, then $\Lambda(u)\subset \Lambda(v)$, that is, $\Lambda(u)$ is a proper
      subset of~$\Lambda(v)$.  If $u$ is neither a descendant nor an ancestor of~$v$, then $\Lambda(u)\cap\Lambda(v)=\varnothing$.
\end{itemize}
A vertex $x\in V(G)$ is \emph{caught by the decomposition} if there exists~$v\in V(T)$ such that $x$ is contained
in the boundary cycle of~$\Lambda(v)$.
The following is a consequence of the proof of a lemma by Asadi \emph{et al.}~\cite[Lemma~2.1]{submany}.
\begin{lemma}\label{lemma-laminar}
Every triangle-free plane graph~$G$ has a $5$-cycle decomposition~$(T,\Lambda)$
such that every vertex of~$G$ that is incident with a $5$-cycle is either incident
with a $5$-face of~$G$ or caught by~$(T,\Lambda)$.
\end{lemma}

\noindent
Combining these results, we obtain the following.
\begin{corollary}\label{cor-decomp}
Let~$\beta\in(0,1/4)$ and let~$n$ be an integer such that every planar
triangle-free graph~$H$ with less than~$n$ vertices has at least~$2^{\beta\abs{V(H)}}$
distinct $3$-colorings.  Set~$d_0=\lfloor 1/\beta\rfloor$ and
$\gamma=\frac{d_0-3}{5(d_0-1)}$.  Let~$G$ be a plane triangle-free graph with $n$
vertices and $s_5$ faces of length~$5$.  If $G$ has less than~$2^{\beta n}$
distinct $3$-colorings, then $G$ has a $5$-cycle decomposition~$(T,\Lambda)$
satisfying $\abs{V(T)}+s_5\ge \gamma n$.
\end{corollary}

\begin{proof}
By Lemma~\ref{lemma-minc}, every vertex of~$G$ of degree at most~$d_0$ is
contained in a $5$-cycle, so in particular $G$ has minimum degree at least~$2$.
Let~$n_0$ be the number of vertices of~$G$ of degree greater than~$d_0$.  Since
$G$ is planar and triangle-free, its average degree is less than~$4$, and thus
$4n>(d_0+1)n_0+2(n-n_0)=2n+(d_0-1)n_0$, and $n_0<\frac{2}{d_0-1}n$.  Hence, $G$
contains more than~$\frac{d_0-3}{d_0-1}n$ vertices of degree at most~$d_0$, which
are all contained in $5$-cycles.  Let~$(T,\Lambda)$ be a $5$-cycle decomposition
obtained by Lemma~\ref{lemma-laminar}.  Note that at most~$5(\abs{V(T)}+s_5)$ vertices
are caught by~$(T,\Lambda)$ or incident with a $5$-face of~$G$, and thus the bound
follows.  \end{proof}

Given a $5$-cycle decomposition~$(T,\Lambda)$ of a graph~$G$ and a vertex~$v\in
V(T)$ with children~$v_1, \dotsc, v_k$ in~$T$, we define~$G_v$ to be the subgraph
of~$G$ drawn in the subset of the plane obtained from the closure of~$\Lambda(v)$
by removing~$\bigcup_{i=1}^k \Lambda(v_i)$.  We say that the decomposition is
\emph{maximal} if for every~$v\in V(T)$, the graph~$G_v$ contains no
separating~$5$-cycle.  A vertex~$v$ of~$V(T)$ is \emph{rich} if either $v$ is the
root of~$T$ or every precoloring of the outer face of~$G_v$ extends to at least
two distinct $3$-colorings of~$G_v$; otherwise, $v$ is \emph{poor}.
These notions are illustrated in Figure~\ref{fig-suburb}.

\begin{figure}[!hbt]
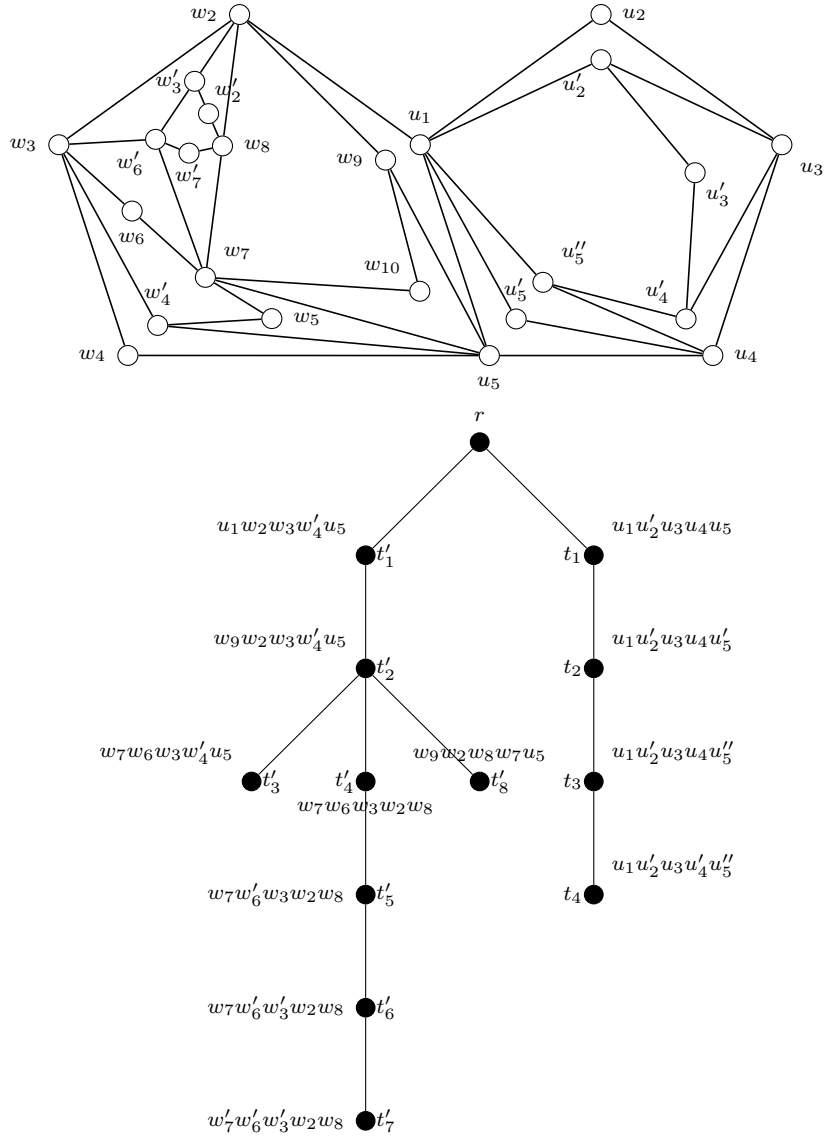

      \begin{center}
\begin{tikzgraph}
\begin{scope}[rotate=72]
\path  (90:2.5)  node[vertex,label=above:$u_1$] (a1) {}
    -- (162:2.5) node[vertex,label=below:$u_5$] (a5) {}
    -- (234:2.5) node[vertex,label=right:$u_4$] (a4) {}
    -- (306:2.5) node[vertex,label=below right:$u_3$] (a3) {}
    -- (18:2.5) node[vertex,label=right:$u_2$] (a2) {};
\draw[edge] (a1)--(a2)--(a3)--(a4)--(a5)--(a1);

\path  (162:1.9) node[vertex,label={[label distance=-.5mm]90:$u_5'$}] (a55) {}
    -- (162:1.3) node[vertex,label=above right:$u_5''$] (a555) {}
    -- (234:1.9) node[vertex,label={[label distance=-.4mm]135:$u_4'$}] (a44) {}
    -- (18:1.9)  node[vertex,label={[label distance=-.8mm]below left:$u_2'$}] (a22) {}
    -- (306:1.3) node[vertex,label={[label distance=-1.5mm]-75:$u_3'$}] (a33) {};

    \draw[edge] (a555)--(a4)--(a55)--(a1)--(a555);
    \draw[edge] (a3)--(a22)--(a1);
    \draw[edge] (a44)--(a33)--(a22);
    \draw[edge] (a555)--(a44)--(a3);

\path (a1)--++(72:2.9377) node[vertex,label=left:$w_2$](w2){}--++(-216:2.9377) node[vertex,label=left:$w_3$](w3){} --++(216:2.9377) node[vertex,label=left:$w_4$](w4){};
\path (w4)--++(.5,-.25) node[vertex,label={[label distance=-.1mm]88:$w'_4$}](w44){}--++(.8,-.4) node[vertex,label=above right:$w_7$](w8){}--++(-.25,-1) node[vertex,label=right:$w_5$](w55){};
\path (w3)--(w8) node[vertex,midway,label=below:$w_6$](w7){} -- (w2) node[vertex,midway,label=right:$w_8$](w9){};
\path (a5)--++(245:-1.25) node[vertex,label=above left:$w_{10}$](w11){};
\path (a1)--++(312:-.5) node[vertex,label=left:$w_{9}$](w10){};
\path (w7)--++(1,0) node[vertex,label={[label distance=-1.2mm]-115:$w'_6$}](w66){} --(w2) node[vertex,midway,label={[label distance=-1.3mm]-182:$w'_{3}$}](w12){};

\path (w66)--(w9) node[vertex,midway,below,label={[label distance=-1.3mm]-90:$w'_7$}](w77){};
\path (w12)--(w9) node[vertex,midway,label={[label distance=-7.5mm]-84:$w'_2$}](w22){};

\draw[edge] (w66)--(w77)--(w9);
\draw[edge] (w12)--(w22)--(w9);
\draw[edge] (w3)--(w66)--(w8);
\draw[edge] (w66)--(w12)--(w2);
\draw[edge] (a1)--(w2)--(w9)--(w8)--(a5)--(w4)--(w3)--(w44)--(w55)--(w8)--(w7)--(w3)--(w2);
\draw[edge] (w2)--(w10)--(a5)--(w44);
\draw[edge] (w8)--(w11)--(w10);
\end{scope}
\end{tikzgraph}
\begin{tikzgraph}
  \node[p,label=above:$r$] {}
    child { node[p,label=above left:$u_1w_2w_3w'_4u_5$](tt1) {}
      child { node[p,label=above left:$w_9w_2w_3w'_4u_5$](tt2) {}
        child { node[p,label=above left:$w_7w_6w_3w'_4u_5$](tt3) {} }
        child { node[p,label=below:$w_7w_6w_3w_2w_8$](tt4) {}
          child { node[p,label=left:$w_7w'_6w_3w_2w_8$](tt5) {} 
            child { node[p,label=left:$w_7w'_6w'_3w_2w_8$](tt7){} 
              child { node[p,label=left:$w'_7w'_6w'_3w_2w_8$](tt8){} } } } }
        child { node[p,label=above:$w_9w_2w_8w_7u_5$](tt6) {} }
} } 
   child {node {} edge from parent[white] }
    child { node[p,label=above right:$u_1u_2'u_3u_4u_5$](t1) {} 
      child{ node[p,label=above right:$u_1u_2'u_3u_4u_5'$](t2) {} 
        child{ node[p,label=above right:$u_1u_2'u_3u_4u_5''$](t3) {}
          child{ node[p,label=above right:$u_1u_2'u_3u_4'u_5''$](t4) {} } } } };
\draw (t1) node[left]{$t_1$};
\draw (t2) node[left]{$t_2$};
\draw (t3) node[left]{$t_3$};
\draw (t4) node[left]{$t_4$};
\draw (tt1) node[right]{$t'_1$};
\draw (tt2) node[right]{$t'_2$};
\draw (tt3) node[right]{$t'_3$};
\draw (tt4) node[left]{$t'_4$};
\draw (tt5) node[right]{$t'_5$};
\draw (tt7) node[right]{$t'_6$};
\draw (tt6) node[right]{$t'_8$};
\draw (tt8) node[right]{$t'_7$};
\end{tikzgraph}
      \end{center}
      \caption{A graph~$G$ (top) along with its maximal $5$-cycle decomposition (bottom):
      for each vertex of the tree, bar the root~$r$, is shown the corresponding separating $5$-cycle of~$G$.
      The rich vertices are the root~$r$ and~$t'_2$; all other vertices of the tree being poor.}\label{fig-suburb}
\end{figure}

\begin{lemma}\label{lemma-poorpieces}
Let~$G$ be a plane triangle-free graph and let~$(T,\Lambda)$ be a maximal
$5$-cycle decomposition of~$G$. If $v\in V(T)$ is poor, then $G_v$ consists of the
$5$-cycle~$K_v$ bounding its outer face and another vertex adjacent to two
vertices of~$K_v$.
\end{lemma}
\begin{proof}
Since $v$ is poor, there exists a $3$-coloring~$\psi$ of~$K_v$ that extends to
a unique $3$-coloring~$\varphi$ of~$G_v$.  Let~$K_v=y_1y_2\dotso y_5$. The definitions
imply that $G_v\neq K_v$.  Thus Lemma~\ref{lemma-2cols} yields that there exists
a vertex $x\in V(G_v)\setminus V(K_v)$ adjacent to two vertices of~$K_v$ of
distinct colors, which can be assumed to be~$y_1$ and~$y_3$.  Since the
decomposition is maximal, the $5$-cycle~$y_1xy_3y_4y_5$ bounds a face of~$G_v$.
If the $4$-cycle~$Q=y_1y_2y_3x$ also bounds a face, then the conclusion of the
lemma holds.  Hence assume that~$Q$ does not bound a face. Because $v$ is poor,
the precoloring of~$Q$ given by~$\varphi$ extends to exactly one $3$-coloring of
the subgraph of~$G_v$ drawn inside~$Q$. So by Lemma~\ref{lemma-2cols}, there
exists a vertex~$x'\in V(G_v)\setminus (V(K_v)\cup \{x\})$ adjacent to two
vertices of~$Q$ with different colors.  Since $\varphi(y_1)\neq\varphi(y_3)$, we have
$\varphi(y_2)=\varphi(x)$ and thus $x'$ is adjacent to~$y_1$ and~$y_3$.  However, this
implies that $G_v$ contains a separating $5$-cycle, namely~$y_1x'y_3y_4y_5$, which
contradicts the assumption that the decomposition~$(T,\Lambda)$ is maximal.
\end{proof}

Lemma~\ref{lemma-poorpieces} implies that in a maximal $5$-cycle
decomposition~$(T,\Lambda)$, each poor vertex of~$T$ has at most one son.  For
a poor vertex~$v$, the \emph{inner face} of~$G_v$ is its $5$-face different from
the outer face.  A path~$P=v_1v_2\ldots v_k$ of poor vertices of~$T$ such that
$v_1$ is the ancestor of all the vertices of the path is called
a \emph{$k$-suburb}.  Let~$G_P=G_{v_1}\cup\dotsb\cup G_{v_k}$, and define the \emph{inner
face of~$G_P$} to be the inner face of~$G_{v_k}$. 
In the example shown in Figure~\ref{fig-suburb}, the path~$P=t_1\dotso t_4$ is a $4$-suburb,
and the graph~$G_P$ is the subgraph of~$G$ induced by~$\{u_1,u_3,u_4,u_5,u'_2,u'_3,u'_4,u'_5,u''_5\}$.
We say that the $k$-suburb~$P$ is
\emph{upwardly mobile} if every precoloring of the outer face of~$G_P$ extends to
at least two distinct $3$-colorings of~$G_P$.
In the example shown in Figure~\ref{fig-suburb},
the path~$P'=t'_4\dotso t'_7$ is a $4$-suburb and it is updwardly mobile; the graph~$G_{P'}$ being the subgraph
of~$G$ induced by~$\{w_2,w_3,w_6,w_7,w_8,w'_2,w'_3,w'_6,w'_7\}$.

Let~$H$ be a plane graph with a plane subgraph~$F$.
A $3$-coloring~$\varphi$ of~$H$ is \emph{rearrangeable} with respect to~$F$ if
there exists a $3$-coloring~$\varphi'$ of~$H$ such that $\varphi'(v)=\varphi(v)$ for all $v\in V(F)$
and some $4$-face of~$H$ is bichromatic in~$\varphi'$.

\begin{lemma}\label{lemma-rearr}
Let~$G$ be a plane triangle-free graph and let~$(T,\Lambda)$ be a maximal $5$-cycle
decomposition of~$G$.  Suppose that $P=v_1v_2\dotso v_{11}$ is an $11$-suburb
in~$T$ and let~$F$ be the union of the boundary cycles of the outer and the inner
face of~$G_P$.  If $P$ is not upwardly mobile, then there exist distinct
non-adjacent vertices~$x$ and~$y$ of~$G_P$ incident with a common $4$-face, such
that every $3$-coloring~$\varphi$ of~$G_P$ that gives to~$x$ and~$y$ the same
color is rearrangeable with respect to~$F$.
\end{lemma}
\begin{proof}
First, we argue that the conclusion of the lemma holds if $G_P$ contains one of the following configurations.
\begin{itemize}
\item[(i)] A vertex~$z\notin V(F)$ of degree two incident with a $4$-face.
\item[(ii)] Two adjacent vertices~$z,z'\notin V(F)$ of degree three, such that $z$ is only incident with $4$-faces.
\item[(iii)] A vertex $z\notin V(F)$ of degree four incident only with $4$-faces, such that two neighbors~$z_1,z_2\notin V(F)$
of~$z$ that are not incident with the same $4$-face at~$z$ have degree three, and $z_1$ is incident only with $4$-faces.
\end{itemize}
In each of these cases, we find two non-adjacent vertices~$x$ and~$y$ incident to
a $4$-face~$f$ in~$G_P$ and next we let~$\varphi$ be an arbitrary $3$-coloring
of~$G_P$ that gives~$x$ and~$y$ the same color.  In case~(i) let~$f=xzyu$ be
a $4$-face incident with~$z$. We can recolor~$z$ with~$\varphi(u)$ so that $f$ is
now bichromatic since $\varphi(x)=\varphi(y)$. In case~(ii), let~$f=zxuy$,
$xzz'x'$, and~$yzz'y'$ be the $4$-faces incident with~$z$.  Since
$\varphi(x)=\varphi(y)$, we can assume that $\varphi(x)=\varphi(y)=1$ and
$\varphi(u)=2$.  Consequently, $\varphi(x')\neq 1\neq \varphi(y')$, and we can
recolor~$z'$ by color~$1$ and $z$ by color~$2$ to make~$f$ bichromatic.  In
case~(iii), let~$zz_1xx'$, $zz_1yy'$, $zz_2x''x'$, and~$zz_2y''y'$ be the
$4$-faces incident with~$z$, and let~$f=xz_1yu$ be the further $4$-face incident
with~$z_1$.  Suppose that $\varphi(x)=\varphi(y)=1$ and $\varphi(u)=2$.  If
$\varphi(z)\neq 2$, then we can recolor~$z_1$ by color~$2$ to make~$f$
bichromatic. If $\varphi(z)=2$, then $\varphi(x')=\varphi(y')=3$ and
$\varphi(x'')\neq 3\neq\varphi(y'')$. Therefore we can recolor~$z_2$ by color~$3$,
$z$ by color~$1$, and~$z_1$ by color~$2$ to make~$f$ bichromatic.

Note that Lemma~\ref{lemma-poorpieces} applies to each of~$v_1,\dotsc,v_{11}$.
For~$i\in\{1,\dotsc,11\}$, let the vertices of the outer face of~$G_{v_i}$ be
labelled $u_1^{i-1}u_2^{i-1}\dotso u_5^{i-1}$ and let the vertices of the inner
face of~$G_{v_{11}}$ be labelled $u_1^{11}u_2^{11}\dotso u_5^{11}$, with the
labels chosen so that for each~$i\in\{1,\dotsc,11\}$, there is a unique
index~$d_i\in\{1,\dotsc,5\}$ such that $u_{d_i}^{i-1}\neq u_{d_i}^i$. Hence,
$u_j^{i-1}=u_j^{i}$ for precisely four values of~$j\in\{1,\ldots,5\}$.

Suppose
that the suburb~$P$ is not upwardly mobile, and let~$\psi_0$ be a precoloring of
its outer face that extends to a unique $3$-coloring~$\psi$ of~$G_P$.  Observe
that for~$i\in\{1,\dotsc,11\}$ the neighbors of~$u_{d_i}^i$ in the outer face
of~$G_{v_i}$ must have different colors, and thus
$\psi(u_{d_i}^i)=\psi(u_{d_i}^{i-1})$.  We conclude that $\psi(u^i_j)=\psi(u^0_j)$
for each~$i\in\{1,\dotsc,11\}$ and each~$j\in\{1,\dotsc,5\}$.

By symmetry, we can assume that $\psi(u^0_1)=1$, $\psi(u^0_2)=2$, $\psi(u^0_3)=3$,
$\psi(u^0_4)=1$, and $\psi(u^0_5)=3$.  It follows that $d_i\in\{1,2,3\}$
for~$i\in\{1,\dotsc,11\}$, hence $u^0_4=\dotsb=u^{11}_4$ and
$u^0_5=\dotsb=u^{11}_5$.  Consider the sequence $D=d_1,\dotsc,d_{11}$. If two
consecutive elements of this sequence are equal, or if $D$ contains a consecutive
subsequence equal to~$1,3,1$ or~$3,1,3$, then $G_P$ contains the
configuration~(i).  If $D$ contains a consecutive subsequence~$a,b,a,b$ for some
distinct~$a,b\in \{1,2,3\}$ with~$\abs{a-b}=1$, then $G_P$ contains the
configuration~(ii).  In both cases, the conclusion of the lemma holds; hence,
assume that no such consecutive subsequences appear in~$D$.  Furthermore, if $D$
contains the consecutive subsequence~$3,1$, then the same graph~$G_P$ arises when
this subsequence is replaced by~$1,3$.  Hence we can assume that $D$ does not
contain the consecutive subsequence~$3,1$, and thus every appearance of~$3$ in~$D$
is followed by~$2$, except possibly for the one in the last position of~$D$.

If $D$ contains the consecutive subsequence~$1,3,2,1,3$ not containing any of the
last two elements of~$D$, then by the previous paragraph $D$ contains, as
a consecutive subsequence, either~$1,3,2,1,3,2,1$ or~$1,3,2,1,3,2,3$. This implies
that $G_P$ contains the configuration~(iii), and so the conclusion of the lemma
holds.  Hence we assume that $D$ does not contain such a consecutive subsequence.

Suppose that $D$ contains a consecutive subsequence~$1,3$, not containing the last
five elements of~$D$.  The next element following $3$ is necessarily~$2$.
The next element cannot be~$3$, as it would be followed by~$2$
and $D$ would contain a consecutive subsequence~$3,2,3,2$.  Hence, the next
element is~$1$ and by the previous paragraph the next one is~$2$, and so $G_P$
contains the configuration~(ii).  It follows that we can assume that $D$ does not
contain a consecutive subsequence~$1,3$ disjoint from the last five elements
of~$D$.  Hence, every appearance of~$1$ not contained in the last six elements
of~$D$ is followed by~$2$.

It follows that $D$ starts with one of the following sequences:
\begin{itemize}
      \item $1,2,3,2,1,2,3,2$;
      \item $2,1,2,3,2,1,2$; or
      \item $2,1,2,3,2,1,3$; or
      \item $2,3,2,1,2,3,2$; or
      \item $3,2,1,2,3,2,1,2$; or
      \item $3,2,1,2,3,2,1,3$.
\end{itemize}
In all the cases, $G_P$ contains the configuration~(ii) or~(iii),
and thus the conclusion of the lemma follows.
\end{proof}

\noindent
We are now ready to demonstrate Theorem~\ref{thm-eq-exp}.
\begin{proof}[Proof of Theorem~\ref{thm-eq-exp}]
We start by showing that (EXP) implies~(REU), for any~$\alpha\in(0,\beta)$.
Fix a planar triangle-free request graph~$(G,R_=,\varnothing,\text{unit})$ with
$n+\abs{R_=}$ vertices. Set~$r=\abs{R_=}$ and
$N=\left\lceil\frac{n(\log_2 3-\beta)}{\beta-\alpha}\right\rceil$.  We can assume that $r\ge
1$.  Every $3$-coloring~$\varphi$ of~$G-R_=$ greedily extends to a $3$-coloring
of~$G$: let~$s(\varphi)$ be the number of requests in~$R_=$ satisfied by any such
extension.  Let~$G'$ be the graph obtained from~$G$ by replacing each vertex
of~$R_=$ by $N$ clones, so $\abs{V(G')}=n+Nr$.  Observe that $\varphi$ extends to
exactly $2^{s(\varphi)N}$ $3$-colorings of~$G'$.  Let~$s_0$ be the maximum
of~$s(\varphi)$ taken over all $3$-colorings $\varphi$ of~$G-R_=$.  As the number
of~$3$-colorings of~$G-R_=$ is at most~$3^n$, it follows that the number
of~$3$-colorings of~$G'$ is at most~$2^{s_0N+n\log_2 3}$.  On the other hand,
(EXP) implies that the number of~$3$-colorings of~$G'$ is at
least~$2^{\beta(n+Nr)}$, and thus
\begin{align*}
      s_0N+n\log_2 3&\ge \beta(n+Nr)\\
      s_0&\ge \beta r - \frac{(\log_2 3-\beta)n}{N}\ge \alpha r.
\end{align*}
Hence, some $3$-coloring~$\varphi$ of~$G-R_=$ extends to a $3$-coloring of~$G$ that
satisfies at least~$\alpha\abs{R_=}$ of the requests, as required.

Next, we show that (REU) implies~(EXP), for~$\beta=\alpha/388$.  Suppose for
a contradiction that there exists a planar triangle-free graph~$G$ with less
than $2^{\beta\abs{V(G)}}$ $3$-colorings. We choose such a graph~$G$
with the least possible number~$n$ of vertices.  Let~$d_0=\lfloor 1/\beta\rfloor$ and
$\gamma=\frac{d_0-3}{5(d_0-1)}$.  Note that $d_0\ge388$, so $\gamma\ge\frac{77}{387}$.
Let~$s_5$ be the number of~$5$-faces of~$G$.  By Corollary~\ref{cor-decomp},
the graph~$G$ has a $5$-cycle decomposition~$(T,\Lambda)$ satisfying $\abs{V(T)}+s_5\ge
\gamma n$, and we can without loss of generality assume that the
decomposition is maximal.  Let~$r$ be the number of rich vertices of~$T$ and
let~$\ell$ be the number of poor leaves of~$T$.  Note that $s_5\ge \ell$.  Let~$S$
be a largest collection of pairwise disjoint $11$-suburbs in~$(T,\Lambda)$. Note that at
most~$10(r+\ell)$ poor vertices of~$T$ belong to no member of~$S$.  Let~$m$
be the number of upwardly mobile suburbs in~$S$, and let~$S_0$ be the subset
of~$S$ consisting of those suburbs that are not upwardly mobile.

For each rich vertex~$v$ and each~upwardly mobile suburb~$P$, every coloring of the outer
face of~$G_v$ and of~$G_P$ extends to at least two $3$-colorings.
Hence, we conclude that $G$ has at least $2^{r+m}$ $3$-colorings, and thus
$r+m<\beta n$.  Hence
\begin{align*}
  \abs{S_0}&\ge\frac{\abs{V(T)}-r-10(r+\ell)-11m}{11}\\
  &=\frac{\abs{V(T)}-11(r+m)-10\ell}{11}\\
  &>\frac{77/387-11\beta}{11}n-s_5.
\end{align*}

Let~$(G',R_=,\varnothing,\text{unit})$ be the request graph obtained from~$G$ by adding,
for every suburb in~$S_0$, a vertex to~$R_=$ adjacent to the two vertices~$x$ and~$y$
obtained from Lemma~\ref{lemma-rearr}.  By~(REU), there exists a $3$-coloring
satisfying $\alpha$-fraction of the requests, and by Lemma~\ref{lemma-rearr}, we
conclude that $G$ has a $3$-coloring with at least~$\alpha\abs{S_0}$ bichromatic
faces.  But then Lemma~\ref{lemma-manycolor} implies that $G$ has more than
$2^{(s_5+\alpha\abs{S_0})/6}\ge 2^{\frac{\alpha(77/387-11\beta)}{66}n}\ge 2^{\beta n}$
$3$-colorings, which is a contradiction.
\end{proof}

\section{Auxiliary results}\label{sec-aux} 
In the rest of the paper, we will use a number of results on coloring and list coloring,
which we present here.
Let us formally state Grötzsch's theorem with one of its extensions.
\begin{theorem}[Gr\"otzsch~\cite{grotzsch1959}, Thomassen~\cite{thom-torus}]\label{thm-grotzsch}
A planar triangle-free graph~$G$ is $3$-colorable.  Moreover, any precoloring of an $(\le\!5)$-cycle
in~$G$ extends to a $3$-coloring of~$G$.
\end{theorem}

Let us recall that Thomassen~\cite{thomassen1995-34} proved the following generalization
of $3$-choosability of planar graphs of girth at least $5$.
\begin{theorem}\label{thm-3choos}
Let~$G$ be a plane graph of girth at least~$5$, let~$P$ be a subpath
of~$G$ drawn in the boundary of the outer face of~$G$ with at most three vertices,
and let~$L$ be an assignment of lists to the vertices of~$G$, satisfying the following
conditions.
All vertices not incident with the outer face have lists of
size three, vertices incident with the outer face not belonging to~$V(P)$
have lists of size two or three, and vertices of~$P$ have lists of size one
giving a proper coloring of~$P$.
If the vertices with list of size two form an independent set, then $G$ is $L$-colorable.
\end{theorem}

Theorem~\ref{thm-3choos} can be strengthened as follows.
\begin{theorem}[Dvořák and Kawarabayashi~\cite{dk}]\label{thm-dvokaw}
Let~$G$ be a plane graph of girth at least~$5$, let~$P=p_1\dotso p_k$ be a subpath
of~$G$ drawn in the boundary of the outer face of~$G$ with $k\le 3$, and
let~$L$ be an assignment of lists to the vertices of~$G$, satisfying the following
conditions.
\begin{itemize}
\item[\textrm{(i)}] All vertices not incident with the outer face have lists of
      size three, vertices incident with the outer face not belonging to~$V(P)$
      have lists of size two or three, and vertices of~$P$ have lists of size one
      giving a proper coloring of~$P$.
\item[\textrm{(ii)}] The graph~$G$ has no path~$v_1v_2v_3$ with
      $\abs{L(v_1)}=\abs{L(v_2)}=\abs{L(v_3)}=2$.
\item[\textrm{(iii)}] The graph~$G$ has no path~$v_1v_2v_3v_4v_5$ with
      $\abs{L(v_1)}=\abs{L(v_2)}=\abs{L(v_4)}=\abs{L(v_5)}=2$ and
      $\abs{L(v_3)}=3$.
\item[\textrm{(iv)}] If $\abs{V(P)}=3$, then at least one endvertex~$p$ of~$P$ is
      contained in no path~$pv_2v_3$ with $\abs{L(v_2)}=\abs{L(v_3)}=2$ and no
      path~$pv_2v_3v_4v_5$ with $\abs{L(v_2)}=\abs{L(v_4)}=\abs{L(v_5)}=2$ and
      $\abs{L(v_3)}=3$.
\end{itemize}
Then $G$ is $L$-colorable.
\end{theorem}

We need the following variant of this result.
If $P$ is a path with $\abs{V(P)}=3$, we call the vertex of~$P$ of degree~$2$ the
\emph{middle vertex} of~$P$.  When $\abs{V(P)}\le 2$, we do not consider any
vertex of~$P$ to be the middle one.
\begin{lemma}\label{lemma-same}
Let~$G$ be a plane graph of girth at least~$5$, let~$P=p_1\dotso p_k$ be a subpath
of~$G$ drawn in the boundary of the outer face of~$G$ with $k\le 3$, and
let~$L$ be an assignment of lists to the vertices of~$G$, satisfying the following
conditions.
\begin{itemize}
\item[\textrm{(i)}] All vertices not incident with the outer face have
      lists~$\{1,2,3\}$, vertices incident with the outer face not belonging
      to~$V(P)$ have lists~$\{1,2\}$ or~$\{1,2,3\}$, and vertices of~$P$ have
      lists of size one giving a proper $3$-coloring of~$P$.
\item[\textrm{(ii)}] The graph~$G$ has no path~$v_1v_2v_3$ with
      $\abs{L(v_1)}=\abs{L(v_2)}=\abs{L(v_3)}=2$.
\item[\textrm{(iii)}] If $\abs{V(P)}=3$, then for one of the endvertices~$p$
      of~$P$, the graph~$G$ contains no path $pv_1v_2$ with $\abs{L(v_1)}=\abs{L(v_2)}=2$.
\end{itemize}
Then $G$ is $L$-colorable.
\end{lemma}
\begin{proof}
We prove the statement by induction, assuming that it holds for all graphs with fewer than~$\abs{V(G)}$ vertices.

We can assume that $G$ is $2$-connected, the cycle~$K$ bounding its outer face has
no chords except for those incident with the middle vertex of~$P$, and there is no
path~$xyz$ such that $x,z\in V(K)$, $y\not\in V(K)$, $x$ is not the middle vertex
of~$P$ and $\abs{L(z)}=2$ --- let us show the last assertion, the other ones
follow similarly.  If $G$ contains such a path, then $G=G_1\cup G_2$ for proper
induced subgraphs~$G_1$ and~$G_2$ with $xyz=G_1\cap G_2$ and $P\subseteq G_1$.  We
$L$-color~$G_1$ by the induction hypothesis, modify the lists of~$x$, $y$ and~$z$
to single-element lists given by this coloring, and extend the coloring to~$G_2$
by the induction hypothesis ($G_2$ satisfies~(iii), since a path~$zv_1v_2$ with
$\abs{L(v_1)}=\abs{L(v_2)}=2$ is forbidden by the assumption~(ii) for~$G$).

We exclude with a similar argument a chord incident with the middle vertex of~$P$:
let~$P=p_1p_2p_3$, where $G$ contains no path~$p_3v_1v_2$ with
$\abs{L(v_1)}=\abs{L(v_2)}=2$.  Write~$G=G_1\cup G_2$ for proper induced
subgraphs~$G_1$ and~$G_2$ intersecting in a chord~$p_2v$, such that $p_3\in
V(G_2)$.  By the induction hypothesis, $G_1$ is $L$-colorable (since it contains
only two vertices~$p_1$ and~$p_2$ with a list of size one).  We modify the list
of~$v$ to the singleton matching this $L$-coloring, and color~$G_2$ by the
induction hypothesis, thereby obtaining an $L$-coloring of~$G$.  Hence, we can
assume that $K$ is an induced cycle.

Next, suppose that $G$ contains a path~$v_1v_2v_3$ with
$\abs{L(v_1)}=\abs{L(v_3)}=2$ and $\abs{L(v_2)}=3$.  By the previous arguments,
$v_1v_2v_3$ is a subpath of~$K$, each neighbor~$u_2$ of~$v_2$ distinct from~$v_1$
and~$v_3$ has a list of size three, and every neighbor of~$u_2$ has a list of size
different from two. Define~$N$ to be the set of neighbors of~$v_2$ distinct
from~$v_1$ and~$v_3$.  Since $G$ has girth greater than $3$, $N$ is in independent
set.  Let~$L'$ be obtained from~$L$ by setting the list of each
vertex in~$N$ to~$\{1,2\}$.  By the induction hypothesis, $G-v_2$ is
$L'$-colorable, and we obtain an $L$-coloring of~$G$ by giving~$v_2$ color~$3$.

Hence, we can assume that $G$ does not contain any such path.  It follows that $G$
and $L$ satisfy the assumptions of Theorem~\ref{thm-dvokaw}, so $G$ is
$L$-colorable.
\end{proof}

We also need the following result on extendability of~$3$-colorings in plane graphs of
girth at least~$5$.

\begin{theorem}[Thomassen~\cite{thomassen-surf}]\label{thm-cycex}
Let~$G$ be a plane graph of girth at least~$5$ with outer face bounded by
a cycle~$K$ of length at most~$9$.  Let~$L$ be an assignment of lists of size one
to vertices of~$K$ yielding a proper coloring of~$K$, and of lists of size three to all
other vertices of~$G$.  If $G$ is not $L$-colorable, then either $\abs{K}\in
\{8,9\}$ and $K$ has a chord, or~$\abs{K}=9$ and a vertex of~$V(G)\setminus V(K)$
has three neighbors in~$K$.
\end{theorem}

Let~$G$ be a plane graph, let~$P$ be a subpath of the boundary of the outer face of~$G$, and let~$X$ be a set of
edges contained in the boundary of the outer face of~$G$ forming a matching vertex-disjoint from~$P$.
Let~$Z$ be the set of vertices of~$G$ incident with~$P$ or an edge in~$X$.
Let~$G'$ be a plane graph such that $G$ is an induced subgraph of~$G'$,
$G'-V(G)$ is an induced cycle~$K$ of length~$|Z|$ bounding the outer face of~$G'$, and
the edges of~$G'$ between~$V(K)$ and~$V(G)$ form a perfect matching between~$V(K)$ and~$Z$.
For each~$z\in Z$, let~$k_z$ be the vertex of~$K$ matched to~$z$.
We say that $G'$ is a \emph{casing} for~$G$, $P$ and~$X$ if for all edges
$xy\in X\cup E(P)$, the vertices~$k_x$ and~$k_y$ are adjacent in~$K$
and the $4$-cycle $k_xxyk_y$ bounds a face of~$G'$.  Let~$p$ be any vertex of~$P$.
For two vertices~$x$ and~$y$ incident with
edges of~$X$, we write $x\prec y$ if $k_x$ precedes~$k_y$ in the clockwise
ordering of vertices of~$K$ starting with~$k_p$.

Let us remark that when $G$ is $2$-connected, its casing is uniquely determined
and the ordering $\prec$ matches the ordering of the vertices around the outer face
of~$G$; casings are just a technical device to enable us to keep track of the order
also when the boundary of the outer face of~$G$ is not a cycle.

We now give one more variation of Theorem~\ref{thm-dvokaw} (note the change
in~(iii), which now permits some paths~$v_1v_2v_3v_4v_5$ with
$\abs{L(v_1)}=\abs{L(v_2)}=\abs{L(v_4)}=\abs{L(v_5)}=2$, as well as the
modifications to~(i) and~(iv)).  In the situations of these theorems, we say that
an edge~$e=xy$ joining two vertices with lists of size two \emph{blocks}
a vertex~$p$ if there exists a path~$puvxy$ with $\abs{L(u)}=2$ and
$\abs{L(v)}=3$.

\begin{lemma}\label{lemma-dvokaw-strong}
Let~$G$ be a plane graph of girth at least~$5$, let~$P=p_1\dotso p_k$ be a subpath
of~$G$ drawn in the boundary of the outer face of~$G$ with $k\le 3$, and let~$L$
be an assignment of lists to vertices of~$G$, satisfying the following conditions.
\begin{itemize}
\item[\textrm{(i')}] All vertices not incident with the outer face have lists of
      size three, vertices incident with the outer face not belonging to~$V(P)$
      have lists of size two or three, and vertices of~$P$ have lists of size one
      giving a proper coloring of~$P$.  Furthermore, each edge of~$G$ that joins
      two vertices with list of size less than three is contained in the boundary
      of the outer face of~$G$.
\item[\textrm{(ii)}] The graph~$G$ has no path $v_1v_2v_3$ with
      $\abs{L(v_1)}=\abs{L(v_2)}=\abs{L(v_3)}=2$.
\item[\textrm{(iii')}] Let~$X$ be the set of edges of~$G$ joining vertices with
      a list of size two.  There exists a casing~$G'$ (with outer face $K$) for~$G$,
      $P$ and~$X$, such that the following holds for the ordering~$\prec$ defined
      by the casing.  If $v_1v_2$ and~$v_4v_5$ are distinct edges of~$X$ with
      $v_1\prec v_2\prec v_4\prec v_5$, then $v_2$ and~$v_4$ have no common
      neighbor, and $v_1$ and~$v_5$ have no common neighbor.
\item[\textrm{(iv')}] If $k=3$, then $G$ contains no
      path~$p_1v_2v_3$ with $\abs{L(v_2)}=\abs{L(v_3)}=2$.  Furthermore, every
      edge $xy\in X$ of~$G$ that blocks~$p_1$ such that $xp_3,yp_3\not\in E(G)$
      also blocks~$p_3$ and satisfies $L(p_2)\subseteq L(x)\cup L(y)$.
\end{itemize}
Then $G$ is $L$-colorable.
\end{lemma}
\begin{proof}
      We prove the statement by induction on~$\abs{V(G)}$, assuming that it holds
      for all graphs with fewer than~$\abs{V(G)}$ vertices.  Clearly, we can
      assume that $G$ is connected.  Also we can assume that $k\ge 2$,
      as otherwise we can add to~$P$ another vertex incident with the outer face
      of~$G$.

Furthermore, we can assume that $G$ is $2$-connected and every chord of the cycle
bounding the outer face of~$G$ is incident with the middle vertex of~$P$:
otherwise, suppose for instance that the outer face of~$G$ has a chord~$xy$ with
neither~$x$ nor~$y$ being the middle vertex of~$P$, and write~$G=G_1\cup G_2$ for
induced subgraphs~$G_1$ and~$G_2$ intersecting in~$xy$ such that $P\subseteq G_1$.
By the induction hypothesis, the graph~$G_1$ has an $L$-coloring~$\varphi_1$ (let
us remark that a casing for~$G_1$, $P$ and~$X_1=X\cap E(G_1)$ postulated by the
assumption~(iii') can be obtained from~$G'$ by removing the vertices
of~$G_2-\{x,y\}$, possibly removing edges between~$x$ or~$y$ and~$K$ if $x$ or~$y$
is not incident with an edge in~$E(P)\cup X_1$, and suppressing vertices of
degree two in~$K$).  Let~$L_2$ be the list assignment obtained from~$L$ by
giving~$x$ and~$y$ singleton lists prescribed by~$\varphi_1$, and find an
$L_2$-coloring of~$G_2$ by the induction hypothesis (letting~$X_2$ be the set of
edges of~$G_2$ joining vertices with list of size two according to~$L_2$, a casing
for~$G_2$, $P_2=xy$ and~$X_2$ can be constructed from~$G'$ by removing the
vertices of~$G_1-\{x,y\}$ and the edges between~$V(G_1)$ and~$V(K)$ not incident
with the edges o~ $X_2$, adding edges~$xk_{p_1}$ and~$yk_{p_2}$, and suppressing
vertices of degree two in~$K$).  This yields an $L$-coloring of~$G$.

A similar argument shows that we can assume the following.
\claim{cl-2chord-spec}{There is no path~$Q=q_1q_2q_3$ of length two with~$q_1$
      and~$q_3$ incident with the outer face of~$G$ and not equal to the middle
      vertex of~$P$, and $q_2$ not incident with the outer face, such that writing
      $G=G_1\cup G_2$ for induced subgraphs~$G_1$ and~$G_2$ with intersection~$Q$
and $P\subseteq G_1$, no neighbor of~$q_1$ in~$G_2$ has a list of size two.}

This implies that $G$ and~$L$ satisfy the assumption~(iii) of
Theorem~\ref{thm-dvokaw}.  Indeed, suppose that $G$ contains a path
$v_1v_2v_3v_4v_5$ with $\abs{L(v_1)}=\abs{L(v_2)}=\abs{L(v_4)}=\abs{L(v_5)}=2$ and
$\abs{L(v_3)}=3$.  By the assumption~(iii') and symmetry, we can assume that
$v_1\prec v_2\prec v_5\prec v_4$.  Since all chords of the outer face are incident
with the middle vertex of $P$, it follows that $v_3$ is not incident with the
outer face.  Let~$G_1$ and~$G_2$ be proper induced subgraphs of~$G$ such that
$G_1\cup G_2=G$, $G_1\cap G_2=v_2v_3v_4$, and $P\subseteq G_1$.  Note that $v_1\in
V(G_1)\setminus V(G_2)$, and by the assumption~(ii) for~$G$, we conclude that
$v_2$ has no neighbor with a list of size two in~$G_2$.  Then the path $v_2v_3v_4$
contradicts~\refclaim{cl-2chord-spec} (with~$q_i=v_{i+1}$ for~$i\in\{1,2,3\}$).

If $G$ and~$L$ satisfy the assumption~(iv) of Theorem~\ref{thm-dvokaw}, it follows
from that theorem that $G$ is $L$-colorable.  Hence, suppose this is not the case.
Thus (iv') implies that $P=p_1p_2p_3$ and $G$ contains an edge~$xy$ joining
vertices with lists of size two that blocks~$p_1$.  Furthermore, (iv') also
implies that either $p_3$ has a neighbor in~$\{x,y\}$ or the edge~$xy$
blocks~$p_3$.  Let~$p_1u_1v_1xy$ with $\abs{L(u_1)}=2$ and $\abs{L(v_1)}=3$ be
a path showing that $xy$ blocks~$p_1$.  Note that $u_1$ has no neighbor with
a list of size two, since we showed in the previous paragraph that $G$ satisfies
the assumption~(iii) of Theorem~\ref{thm-dvokaw}.  By~\refclaim{cl-2chord-spec}
and the absence of chords not incident with~$p_2$, we conclude that $p_1u_1v_1xy$
is contained in the boundary of the outer face of~$G$.  By a symmetric argument
at~$p_3$, we conclude that the outer face of~$G$ is bounded by either
a $7$-cycle~$p_1u_1v_1xyp_3p_2$ or a $9$-cycle~$p_1u_1v_1xyv_3u_3p_3p_2$ with
$\abs{L(u_3)}=2$ and $\abs{L(v_3)}=3$.  By Theorem~\ref{thm-cycex}, we conclude
that $G$ is $L$-colorable, unless its outer face is bounded by a $9$-cycle and $G$
contains a vertex~$z$ adjacent to~$p_2$, $v_1$, and~$v_3$.  However, in that case
$G$ is $L$-colorable as well, since $L(p_2)\subseteq L(x)\cup L(y)$ by the
assumption~(iv').
\end{proof}

Finally, we consider distance colorability of planar triangle-free graphs.  The
\emph{Clebsch graph} is the graph with vertex set equal to the elements of the
finite field~GF(16) and edges joining two elements if their difference is
a perfect cube.
\begin{theorem}[Naserasr~\cite{homclebsch}]\label{thm-clebsch}
Every planar triangle-free graph has a homomorphism to the Clebsch graph.
\end{theorem}

\noindent
Since the Clebsch graph is triangle-free, Theorem~\ref{thm-clebsch} has the
following consequence, also noted by Naserasr~\cite{homclebsch}.
\begin{corollary}\label{cor-dist13}
Every planar triangle-free graph has a proper coloring by~$16$ colors such that
any two vertices joined by a path of length~$3$ have different colors.
\end{corollary}

\section{Requests at a vertex}\label{sec:atvert} 
In this section, we consider the case of a request graph with only non-equality
requests and all requests adjacent to one vertex~$v$.  Let~$T$ be the set of
vertices other than~$v$ adjacent to the requests and let~$S$ be the set of
non-request neighbors of~$v$.  We can without loss of generality assign to~$v$
color~$3$, and thus we equivalently ask for all vertices of~$S$ as well as
a constant fraction of the vertices of~$T$ to be colored from the list~$\{1,2\}$.
After removing~$v$ and the request vertices, the vertices of~$S\cup T$ will be
incident with a single face of the graph, say the outer one.  If the request graph
had girth at least~$5$ and $S=\varnothing$, we could satisfy all requests in any
independent subset of~$T$ using Theorem~\ref{thm-3choos}, and this would allow us
to satisfy at least $1/3$-fraction of all the requests.  However, the graphs is
only assumed to be triangle-free, and thus a more involved argument is needed.

Let us introduce a definition motivated by the situation described in the previous
paragraph.  Let~$G$ be a graph, let~$S$ and $T$ be disjoint subsets of its
vertices, let~$P$ be a path in~$G$ disjoint from~$S\cup T$, and let~$w\colon T\to
\QQ^+$ be an assignment of positive weights to the vertices in~$T$.  If $S$ is an
independent set in~$G$, we say that $C=(G,P,S,T,w)$ is a \emph{cog}, and the
elements of~$T$ are its \emph{demands}.  A \emph{$3$-coloring} of the cog is
a $3$-coloring~$\varphi$ of~$G$ such that $\varphi(v)\in\{1,2\}$ for all~$v\in S$.
For a real number~$\alpha$, we say that $\varphi$ \emph{satisfies
$\alpha$-fraction of demands} if $w(\varphi^{-1}(\{1,2\})\cap T)\ge \alpha w(T)$.
We say that the cog is \emph{plane} if $G$ is a plane graph, $P$ is a subpath of
the boundary of the outer face of~$G$, and $S$ and~$T$ consist only of vertices
incident with the outer face of~$G$.  The \emph{girth} of the cog is defined as
the length of the shortest cycle in~$G$.

In all forthcoming figures, vertices of~$P$ are depicted by filled circles,
vertices of~$S$ are depicted by squares, vertices of~$T$ are depicted by squares
containing a question mark, and all other vertices are depicted by empty circles.

Let~$C=(G,P,S,T,w)$ be a plane cog and let~$Q$ be an induced path in~$G$ such that
the ends of~$Q$ are incident with the outer face and no other vertex or edge
of~$Q$ is incident with the outer face.  Then $G=G_1\cup G_2$ for proper induced
subgraphs $G_1$ and $G_2$ with intersection~$Q$.  Suppose that $P\subseteq G_1$,
and define $C_1=(G_1,P,S\cap V(G_1),T\cap V(G_1), w\restriction (T\cap V(G_1)))$,
and $C_2=(G_2,Q,S\cap V(G_2)\setminus V(Q),T\cap V(G_2)\setminus V(Q),
w\restriction (T\cap V(G_2)\setminus V(Q)))$.  We say that $C_1$ and $C_2$ are the
\emph{$Q$-components of~$C$}, and that $C_2$ is \emph{cut off} by~$Q$.  If $Q$ has
length~$2$ and one of its ends belongs to~$S\cup T$, we say that $Q$ is
a \emph{weak $2$-chord}.  A cog $C'=(G',P',S',T',w')$ is
a \emph{subcog} of~$C$ if $G'\subseteq G$, $P'=P\cap G'$, $S'\subseteq S\cap
V(G')$, $T'\subseteq T\cap V(G')$, and $w'$ is the restriction of~$w$ to~$T'$.

We observe that Theorem~\ref{thm-3choos} implies that if $C=(G,P,S,T,w)$ is
a plane cog of girth at least~$5$ with $\abs{V(P)}\le 3$, then every $3$-coloring
of~$P$ extends to a $3$-coloring of the cog.  In Lemma~\ref{lemma-full}, we extend
this to show that when $\abs{V(P)}=2$ (and with a few exceptions), such
a $3$-coloring can satisfy a constant fraction of the demands, even if the cog has
girth $4$.  This directly implies the result for request graphs with only
non-equality requests at a single vertex, Corollary~\ref{cor-single}.

A plane cog~$(G,P,S,T,w)$ is \emph{polished} if $T$ is an independent set and $G$
does not contain a path~$v_1v_2v_3$ with $v_1,v_3\in T$ and~$v_2\in S$.
Let us first deal with the special case of satisfying demands in polished cogs
of girth at least five.

\begin{figure}
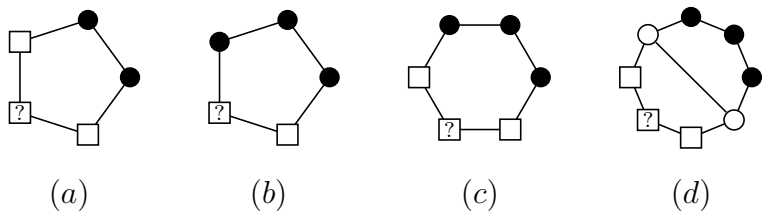

      \begin{center}
      \begin{tikzgraph}
            \draw[edge] (0:\R) node[p]{} -- (72:\R) node[p]{} -- (144:\R) node[s]{} -- (216:\R) node[s]{?} -- (288:\R) node[s]{} -- cycle (270:\R) node[below=\R/2]{{\normalsize$(a)$}};
\end{tikzgraph}
\hspace{1.5em}
      \begin{tikzgraph}
            \draw[edge] (0:\R) node[p]{} -- (72:\R) node[p]{} -- (144:\R) node[p]{} -- (216:\R) node[s]{?} -- (288:\R) node[s]{} -- cycle (270:\R) node[below=\R/2]{{\normalsize$(b)$}};
\end{tikzgraph}
\hspace{1.5em}
      \begin{tikzgraph}
      \draw[edge] (0:\R) node[p]{} -- (60:\R) node[p]{} -- (120:\R) node[p]{} -- (180:\R) node[s]{} -- (240:\R) node[s]{?}
      -- (300:\R) node[s]{}  -- cycle (270:\R) node[below=\R/2]{{\normalsize$(c)$}};
\end{tikzgraph}
\hspace{1.5em}
      \begin{tikzgraph}
            \draw[edge] (0:\R) node[p]{} -- (45:\R) node[p]{} -- (90:\R) node[p]{} -- (135:\R) node[vertex](u){} -- (180:\R) node[s]{} -- (225:\R) node[s]{?} -- (270:\R) node[s]{} -- (315:\R) node[vertex](v){} -- cycle (270:\R) node[below=\R/2]{{\normalsize$(d)$}}; \draw[edge] (u)--(v);
\end{tikzgraph}
\end{center}
\caption{Obstructing cogs.}\label{fig-obstr}
\end{figure}

\begin{lemma}\label{lemma-polished}
Let~$\alpha_1=1/562$.  Let~$C=(G,P,S,T,w)$ be a polished plane cog of girth at least~$5$, where $\abs{V(P)}\le 3$.
Let~$\psi$ be a $3$-coloring of~$P$.
If $C$ does not contain any of the subcogs depicted in Figure~\ref{fig-obstr}, then $\psi$ extends to a $3$-coloring of~$C$
satisfying $\alpha_1$-fraction of the demands.
\end{lemma}
\begin{proof}
Suppose on the contrary that $C$ and~$\psi$ form a counterexample
with~$\abs{V(G)}$ as small as possible.  Clearly, $G$ is connected and vertices
not belonging to~$S\cup T\cup V(P)$ have degree at least three.  

Also, $G$ is $2$-connected: otherwise, let~$v$ be a cutvertex of~$G$.  If $v$ is
not the middle vertex of~$P$, then let~$C_1$ and $C_2$ be the $v$-components
of~$C$.  Note that neither $C_1$ nor~$C_2$ contains a subcog depicted in
Figure~\ref{fig-obstr}. By the minimality of~$C$, the precoloring~$\psi$ extends
to a $3$-coloring~$\varphi_1$ of~$C_1$ satisfying $\alpha_1$-fraction of its
demands.  Furthermore, the $3$-coloring of~$v$ by color~$\varphi_1(v)$ extends to
a $3$-coloring~$\varphi_2$ of~$C_2$ satisfying $\alpha_1$-fraction of its demands.
The combination of~$\varphi_1$ and $\varphi_2$ is a $3$-coloring of~$C$ satisfying
$\alpha_1$-fraction of its demands, which contradicts the assumption that
$(C,\psi)$ is a counterexample.  A similar argument excludes the case that $v$ is
the middle vertex of~$P$ and thus $G$ contains no cutvertices.  In particular, the
outer face of~$G$ is bounded by a cycle~$K$.  Similarly, Theorem~\ref{thm-cycex}
implies the following.
\claim{cl-nosep}{Every cycle in~$G$ of length at most~$7$ bounds a face,
and the open disk bounded by any $8$-cycle in~$G$ contains no vertices.}

Suppose that $K$ has a chord~$uv$. Let us first consider the case that neither~$u$
nor~$v$ is the middle vertex of~$P$.  Let~$C_1$ and $C_2$ be the
$uv$-components of~$C$, and let~$G_2$ be the graph of~$C_2$.  Note that $C_1$ does
not contain a subcog depicted in Figure~\ref{fig-obstr}, so the induction
hypothesis ensures that $\psi$ extends to a $3$-coloring of~$C_1$. Considering now
$C_2$ with~$u$ and~$v$ precolored as prescribed by this extension, we deduce that
that $C_2$ must contain the subcog depicted in Figure~\ref{fig-obstr}(a)---if
$C_2$ did not contain such a subcog, we obtain a contradiction as in the previous
paragraph, since $C_2$ has only two precolored vertices.  Hence,
$G_2$ contains a path $ux_1x_2x_3v$ with $x_1,x_3\in S$ and $x_2\in T$.  Since $C$
is polished, $u,v\not\in S\cup T$.  We obtain the following.
\claim{cl-chord1}{The cycle~$K$ has no chord with an end in~$S\cup T$, unless the
      other end of the chord is the middle vertex of~$P$.}

In particular, the edges $ux_1$, $x_1x_2$, $x_2x_3$, and~$x_3v$ are not chords,
and since every $5$-cycle in~$G$ bounds a face by~\refclaim{cl-nosep}, we conclude
that $G_2$ is equal to the $5$-cycle~$vux_1x_2x_3$.
\claim{cl-chord2}{If $uv$ is a chord of the cycle~$K$ not incident with the middle
      vertex of~$P$, then the $uv$-component of~$C$ cut off by~$uv$ is the cog
      depicted in Figure~\ref{fig-obstr}(a).}

\refclaim{cl-chord1} implies that each vertex of~$T$ is
incident with at most two vertices of~$S$ (consecutive to it in~$K$). Since $C$ is polished, each
component of~$G[S\cup T]$ is a path of length at most two contained in~$K$, and if
its length is two, then its middle vertex belongs to~$T$.  We next show the following.
\claim{cl-2chord}{Suppose that $Q=uvz$ is a weak $2$-chord of~$C$, where $z\in
      S\cup T$ and $u$ is not the middle vertex of~$P$.  Then the
      $Q$-component~$C'$ of~$C$ cut off by~$Q$ is equal to the cog depicted in
      Figure~\ref{fig-obstr}(b), and since $C$ is polished, it follows that
      $u\not\in S\cup T$ and $z\in S$.}

Suppose for a contradiction that this is not the case, and let~$Q=uvz$ be a weak
$2$-chord satisfying the assumptions that fails the conclusion
of~\refclaim{cl-2chord} with~$C'$ minimal.  As before, we argue that $C'$ contains
a subcog $C''$ depicted in Figure~\ref{fig-obstr}.  If $C''$ is the subcog from
Figure~\ref{fig-obstr}(a), then since $C$ is polished, $C''$ contains the
edge~$uv$ (and not~$vz$).  Let~$u'\in S$ be the neighbor of~$v$ in~$C''$ distinct
from~$u$.  However, then the cut-off~$u'vz$-component of~$C$ contradicts the
minimality of~$C'$ (it cannot be equal to the cog depicted in
Figure~\ref{fig-obstr}(b) since $C$ is polished and $u',z\in S\cup T$).
Similarly, as $C$ is polished, $C''$ is not the cog depicted in
Figure~\ref{fig-obstr}(c).  If $C''$ is the cog depicted in
Figure~\ref{fig-obstr}(b), then \refclaim{cl-nosep} and~\refclaim{cl-chord1} yield
that $C'=C''$, which contradicts the definition of~$Q$.

\begin{figure}
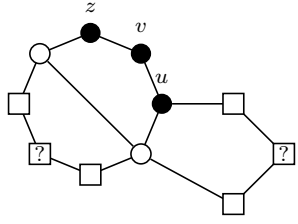

      \begin{center}
      \begin{tikzgraph}
            \draw[edge] (0:\RR) node[p,label=above:$u$](x){} -- (45:\RR) node[p,label=above:$v$]{} -- (90:\RR) node[p,label=above:$z$]{} -- (135:\RR) node[vertex](u){} -- (180:\RR) node[s]{} -- (225:\RR) node[s]{?} -- (270:\RR) node[s]{} -- (315:\RR) node[vertex](v){} -- cycle ; \draw[edge] (u)--(v); \draw[edge] (x) -- ++(0:\RR) node[s]{} -- ++(-45:\RR) node[s]{?} -- ++(225:\RR) node[s]{} -- (v);
\end{tikzgraph}
\end{center}
\caption{A cog split off by a weak $2$-chord.}\label{fig-combincog}
\end{figure}
Finally, suppose that $C''$ is the cog depicted in Figure~\ref{fig-obstr}(d).  As
$C$ is polished, the minimality of~$C'$ along with~\refclaim{cl-nosep}
and~\refclaim{cl-chord2} imply that either $C'=C''$ or $C'$ is the cog depicted in
Figure~\ref{fig-combincog}.  Let~$\beta$ be the weight of the unique demand
of~$C''$.  Let~$C_1=(G_1,P,S_1,T_1,w_1)$ be the $Q$-component of~$C$ distinct
from~$C'$.  If $z\in S$, then let~$C'_1=C_1$; otherwise (when $z\in T$),
let~$C'_1$ be obtained from~$C_1$ by increasing the weight of~$z$ by~$\beta$.  By
the minimality of~$C$, any $3$-coloring of~$P$ extends to a $3$-coloring~$\varphi$
of~$C'_1$ satisfying $\alpha_1$-fraction of its demands.  If $\varphi(u)\neq 3$,
then we can color the neighbor of~$u$ in~$C'$ with a list of size three by
color~$3$ and extend the coloring so that all demands in~$C'$ are satisfied, and
the resulting $3$-coloring satisfies $\alpha_1$-fraction of demands
of~$C$.  Hence, suppose that $\varphi(u)=3$.  If $z\in T$ and $\varphi(z)=3$ (so
that the demand of~$z$ is not satisfied), then we extend $\varphi$ to~$C''$
without satisfying its unique demand; otherwise $\varphi(z)\in \{1,2\}$, and we
observe that $\varphi$ can be extended to a $3$-coloring of~$C''$ satisfying its
demand.  In either case, if $C'\neq C''$, then the coloring extends to
a $3$-coloring of~$C'$ satisfying the demand of~$C'$ not in~$C''$, since
$\varphi(u)=3$.  Observe that in all the cases, the resulting $3$-coloring of~$C$
satisfies $\alpha_1$-fraction of its demands.  This is a contradiction, showing
that \refclaim{cl-2chord} holds.

\smallskip

\begin{figure}
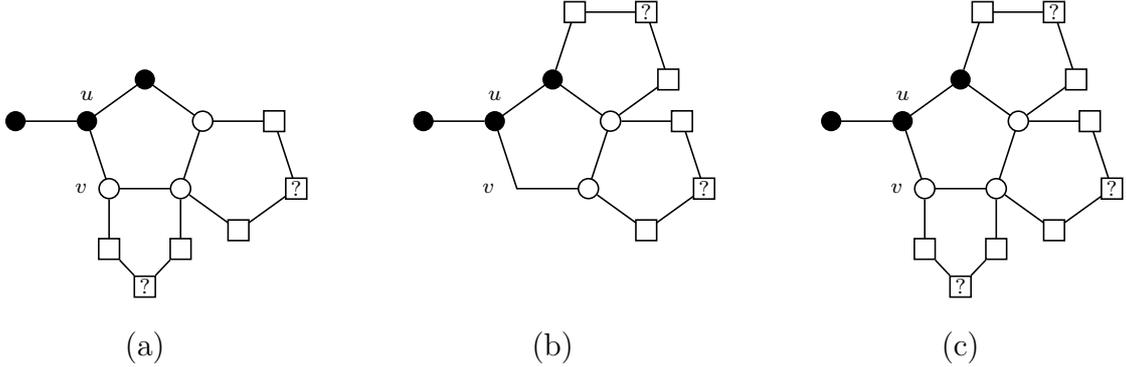

      \begin{center}
\begin{tikzgraph}
     (0,0) node(u){};
   \draw[edge] (90:\R) node[p]{} -- (162:\R) node[p,label=above:$u$](o){} -- (234:\R) node[n,label=left:$v$](w){} -- (306:\R) node[n](v){} -- (378:\R) node[n](u){} -- cycle (270:\R) node[below=2*\R]{{\normalsize(a)}};
   \draw[edge] (u) -- ++(0:\RR) node[s]{} -- ++(-72:\RR) node[s]{?} -- ++(-144:\RR) node[s]{} -- (v);
   \draw[edge] (w) -- ++(270:\R) node[s](ww){};
   \draw[edge] (v) -- ++(270:\R) node[s](vv){};
   \node[draw,edge,s] at ($(ww)!0.5!(vv)+(0,-.5)$)(l){?};
\draw[edge] (ww)--(l)--(vv);
   \draw[edge] (o) -- ++(-\RR,0) node[p]{};
\end{tikzgraph}
\hfill
\begin{tikzgraph}
     (0,0) node(u){};
   \draw[edge] (90:\R) node[p](top){} -- (162:\R) node[p,label=above:$u$](o){} -- (234:\R) node[label=left:$v$](w){} -- (306:\R) node[n](v){} -- (378:\R) node[n](u){} -- cycle (270:\R) node[below=2*\R]{\normalsize(b)};
   \draw[edge] (u) -- ++(0:\RR) node[s]{} -- ++(-72:\RR) node[s]{?} -- ++(-144:\RR) node[s]{} -- (v);
   
   \draw[edge] (o) -- ++(-\RR,0) node[p]{};
   \draw[edge] (top) -- ++(72:\RR) node[s]{} -- ++(0:\RR) node[s]{?} -- ++(-72:\RR) node[s]{} -- (u);
\end{tikzgraph}
\hfill
\begin{tikzgraph}
     (0,0) node(u){};
   \draw[edge] (90:\R) node[p](top){} -- (162:\R) node[p,label=above:$u$](o){} -- (234:\R) node[n,label=left:$v$](w){} -- (306:\R) node[n](v){} -- (378:\R) node[n](u){} -- cycle (270:\R) node[below=2*\R]{\normalsize(c)};
   \draw[edge] (u) -- ++(0:\RR) node[s]{} -- ++(-72:\RR) node[s]{?} -- ++(-144:\RR) node[s]{} -- (v);
   \draw[edge] (w) -- ++(270:\R) node[s](ww){};
   \draw[edge] (v) -- ++(270:\R) node[s](vv){};
   \node[draw,edge,s] at ($(ww)!0.5!(vv)+(0,-.5)$)(l){?};
\draw[edge] (ww)--(l)--(vv);

   \draw[edge] (o) -- ++(-\RR,0) node[p]{};
   \draw[edge] (top) -- ++(72:\RR) node[s]{} -- ++(0:\RR) node[s]{?} -- ++(-72:\RR) node[s]{} -- (u);
\end{tikzgraph}

      \end{center}
      \caption{Compositions of the cog~(d) with cogs~(a) from Figure~\ref{fig-obstr}.}\label{fig-d-compose}
\end{figure}

Suppose now that $\abs{V(P)}=3$ and $K$ has a chord~$uv$, where~$u$ is the middle
vertex of~$P$.  Let~$G_1$ and $G_2$ be proper induced subgraphs of~$G$ such that
$G=G_1\cup G_2$ and $uv=G_1\cap G_2$.  For~$i\in\{1,2\}$, let~$P_i$ be the path
in~$G_i$ consisting of~$uv$ and an edge of~$P$; let~$C_i=(G_i,P_i,S\cap
V(G_i)\setminus \{v\}, T\cap V(G_i)\setminus \{v\}, w\restriction (T\cap
V(G_i)\setminus \{v\}))$.  If $C_j$, for some~$j\in\{1,2\}$, does not contain any
of the subcogs depicted in Figure~\ref{fig-obstr}, then let
$C'_{3-j}=(G_{3-j},P\cap G_{3-j},S\cap V(G_{3-j}), T\cap V(G_{3-j}), w\restriction
(T\cap V(G_{3-j})))$, extend~$\psi$ to a $3$-coloring of~$C'_{3-j}$ satisfying
$\alpha_1$-fraction of its demands by the minimality of~$C$, extend the resulting
precoloring of~$P_j$ to a $3$-coloring of~$C_j$ satisfying $\alpha_1$-fraction of
its demands by the minimality of~$C$, and obtain a contradiction as before.
Hence, we can assume that for each~$i\in\{1,2\}$, the cog~$C_i$ contains one of
the subcogs depicted in Figure~\ref{fig-obstr}.  If $C_i$ contains one of the
subcogs~(b), (c), or~(d) from that figure, it is actually equal to it
by~\refclaim{cl-nosep}, \refclaim{cl-chord1} and~\refclaim{cl-2chord}, with the
exception of the subcog~(d), which can have copies of subcog~(a) attached to two
of its edges (see Figure~\ref{fig-d-compose}).  If $C_i$ contains the
subcog~$C'_i$ equal to~(a) from the figure, then since $C$ does not contain such
a subcog, we conclude that $C'_i$ contains the edge~$uv$ (and not the edge
of~$P$).  But then $G_i$ contains another chord incident with~$u$, and we can
repeat the same argument (at most once, since this chord is incident with a vertex
in~$S$ and thus cannot be followed by another copy of the cog depicted in
Figure~\ref{fig-obstr}(a)).

In conclusion, if $uv_1, \dotsc, uv_m$ are all chords incident with~$u$ in cyclic
order around~$u$, then $m\le 3$ and $C$ consists of~$P=p_1up_2$, these chords,
a path~$v_1x_1y_1v_2$ if $m=2$ and $v_1x_1y_1v_2y_2x_2v_3$ if $m=3$, with
$y_1,y_2,v_1,v_3\in S$ and $x_1,x_2\in T$, and subcogs depicted in
Figure~\ref{fig-obstr} (b), (c), or~(d) or Figure~\ref{fig-d-compose} attached to
the paths~$p_1uv_1$ and~$p_2uv_m$.  Note that if $m\ge 2$, then the demands~$x_1,
\dotsc, x_{m-1}$ can be satisfied by giving the vertices~$v_1, \dotsc, v_m$
alternating colors different from~$\psi(u)$, and if $m=1$ and $v_1\in T$, then we
can always satisfy the demand of~$v_1$ by giving it a color
in~$\{1,2\}\setminus\{\psi(u)\}$.  Similarly, at least a $2/3$-fraction of the
demands in each of the two subcogs at the ends can be satisfied with the proper
choice of color of~$v_1$ or~$v_m$ (if say $v_1\in S$ so that its color may be
forced by~$\psi$, then since $C$ is polished and does not contain the subcog (a),
it follows that the subcog cut off by~$p_1uv_1$ is either~(d) or the one 
depicted in Figure~\ref{fig-d-compose}(b); and for these, it suffices that $v_1$ will
be colored by~$1$ or~$2$ to enable us to satisfy its demands).  We conclude that
every $3$-coloring of~$P$ extends to a $3$-coloring of~$C$ satisfying
$1/4$-fraction of its demands.  This is a contradiction, showing the following.
\claim{cl-chord3}{No chord of~$K$ is incident with the middle vertex of~$P$.}

Suppose that a vertex $p\in V(P)$ is incident with a chord~$Q$, and let~$C'$ be
the $Q$-component of~$C$ cut off by~$Q$.  By \refclaim{cl-chord2}, $C'$ is the
graph depicted in Figure~\ref{fig-obstr}(a).  If $\psi(p)=3$, then observe that
any $3$-coloring of~$C$ can be modified by recoloring within~$C'$ so that the
demand of~$C'$ is satisfied.  Hence, the minimality of~$C$ implies the following.
\claim{cl-chord4}{If a chord of~$K$ is incident with a vertex~$p\in V(P)$, then
      $\psi(p)\in \{1,2\}$.}

Note that we can assume that $\abs{V(P)}\ge 2$, as otherwise we can include
another vertex in~$P$ without creating the subcog depicted in
Figure~\ref{fig-obstr}(a).  Next, we prove the following.
\claim{cl-gap1}{Let~$v_1v_2v_3$ be a path of~$G$ with $v_1,v_3\in S\cup T$ and
      $v_2\not\in T$.  Then $v_1v_2v_3$ is a subpath of~$K$.  Furthermore,  if
      $v_1,v_3\in S$, then $v_2$ is either incident with a chord or a weak
      $2$-chord of~$K$ (together with \refclaim{cl-chord2},
      \refclaim{cl-2chord}, and \refclaim{cl-chord3}, this implies that $v_1$
or~$v_3$ is an endvertex of a path of length two in~$G[S\cup T]$).}

Suppose for a contradiction that this is not the case.  Note that $v_1v_2v_3$ is
a subpath of~$K$ by \refclaim{cl-chord1}, \refclaim{cl-2chord}, and
\refclaim{cl-chord3}, and $v_2\not\in V(P)$ since $\abs{V(P)}\ge 2$. Assume that $v_1$ and~$v_3$ belong to~$S$,
and that $v_2$ is neither incident with a chord nor a weak $2$-chord of~$K$.
Let~$N$ be the set of neighbors of~$v_2$ distinct from~$v_1$ and~$v_3$.  Since $v_2$ is not
incident with a chord, no vertex of~$N$ belongs to~$S\cup T\cup V(P)$.  Since
$v_2$ is not incident with a weak $2$-chord, no vertex in~$N$ is adjacent to
a vertex in~$S\cup T$.  Since $G$ is triangle-free, $N$ is an independent set.
Hence, $C'=(G-v_2,P,S\cup N,T,w)$ is a polished cog.  If $C'$ does not contain any
of the subcogs depicted in Figure~\ref{fig-obstr}, then it follows from the
minimality of~$C$ that $\psi$ extends to a $3$-coloring of~$C'$ satisfying
$\alpha_1$-fraction of its demands, which can be extended to a $3$-coloring
of~$C$ by giving~$v_2$ the color~$3$. This contradicts the assumption that $C$ is
a counterexample.  Hence $C'$ contains a subcog~$C''$ depicted in
Figure~\ref{fig-obstr}.  Clearly, $C''$ contains a vertex~$y\in N$.  Furthermore,
$y$ has a neighbor~$z$ in~$C''$ that belongs to~$T$.  It follows that either
$v_2y$ is a chord or $v_2yz$ is a weak $2$-chord of~$K$, a contradiction which
establishes~\refclaim{cl-gap1}.

\smallskip
  
Without loss of generality, we can assume that $G[S\cup T]$ contains no isolated
vertices belonging to~$T$, as these can be moved into~$S$.  Let~$T_1$ and~$T_2$ be
the vertices of~$T$ belonging to paths of lengths~$1$ and~$2$ in~$G[S\cup T]$,
respectively.

Suppose that $w(T_1)\ge 2\alpha_1 w(T)$.  We let~$t_1, \dotsc, t_n$ be the
vertices of~$T_1$ in order around~$K$, where~$P$ is between~$t_n$ and~$t_1$;
without loss of generality, $w(t_1)\ge w(t_n)$.  Let~$T'_1=T_1$ if $n=1$ and
$T'_1=T_1\setminus\{t_n\}$ otherwise; we have $w(T'_1)\ge w(T_1)/2$.  Let~$L$ be
the list assignment for~$G$ such that 
\[
      L(v)=\begin{cases}
            \{\psi(v)\}&\quad\text{if~$v\in V(P)$,}\\
            \{1,2\}&\quad\text{if~$v\in S\cup T'_1$,}\\
            \{1,2,3\}&\quad\text{otherwise.}
      \end{cases}
\]
An $L$-coloring of~$G$ would yield a $3$-coloring of~$C$ that satisfies all demands
in~$T'_1$, with weight at least~$w(T_1)/2\ge \alpha_1 w(T)$. This would contradict the
assumption that $C$ is a counterexample.  Therefore, $G$ is not $L$-colorable, and
thus it violates one of the assumptions of Lemma~\ref{lemma-same}.  The
assumptions~(i) and~(ii) are clearly satisfied.  Hence, the assumption~(iii) is
violated, so $G$ contains a walk~$v_1v_2p_1p_2p_3v_3v_4$ (where $P=p_1p_2p_3$)
with $\abs{L(v_1)}=\abs{L(v_2)}=\abs{L(v_3)}=\abs{L(v_4)}=2$; i.e.,
$v_1,\dotsc,v_4\in S\cup T_1'$. Consequently, \refclaim{cl-chord1} ensures that
this walk is a subwalk of~$K$, and thus it contains both~$t_1$ and~$t_n$.  Hence,
$t_1,t_n\in T'_1$, and thus $n=1$ and $v_1=v_3$ and $v_2=v_4$.  But then $C$
contains the subcog depicted in Figure~\ref{fig-obstr}(b).  This is
a contradiction, showing that the following holds.
\claim{cl-T1}{We have $w(T_1)<2\alpha_1 w(T)$.}

We also note the following direct corollary of~\refclaim{cl-gap1}.
\claim{cl-gap1mod}{Let~$v_1v_2v_3$ be a path of~$G$ with $v_1,v_3\in S\cup T_2$
and $v_2\not\in T_2$.  Then $v_1v_2v_3$ is a subpath of~$K$, $v_1,v_3\in S$, and
$v_2$ is either incident with a chord or a weak $2$-chord of~$K$.}

A vertex~$z\in T_2$ is \emph{peripheral} if there exists either a chord or a weak
$2$-chord~$Q$ such that $z$ is contained in the $Q$-component~$C_z$ of~$C$ cut off
by~$Q$, and at
least one of the endvertices of~$Q$ is adjacent to a vertex in~$S$ not belonging
to~$C_z$.  We choose one of the endvertices of~$Q$ with this property and call it
the \emph{connector} of~$z$.
Note that \refclaim{cl-chord2}, \refclaim{cl-2chord}
and~\refclaim{cl-chord3} imply that the graph of~$C_z$ is a $5$-cycle.

Let~$T_p$ be the set of peripheral vertices and suppose that $w(T_p)\ge
48\alpha_1w(T)$.  Let~$Y$ be the set of connectors of the peripheral vertices, and
for~$y\in Y$, let us define
\[
      \omega(y)=\sum_{\substack{\text{$z\in T_p$}\\\text{with connector~$y$}}} w(z).
\]
Note that $\omega(Y)=w(T_p)$.  By Corollary~\ref{cor-dist13}, there exists an
independent set~$Y'\subseteq Y$ such that no two vertices of~$Y'$ are joined by
a path of length~$3$ in~$G$ and $\omega(Y')\ge w(T_p)/16$.  Let~$y_1, \dotsc, y_n$
be the vertices of~$Y'$ in order around~$K$, with~$P$ being contained between~$y_n$
and~$y_1$.  We consider the cycle~$y_1\dotso y_n$ built on~$Y'$ and we let~$Y''$
be an independent set in this cycle such that~$\omega(Y'')\ge\omega(Y')/3$.

Let~$G_0$ be the subgraph of~$G$ obtained by removing the vertices in~$T_p$ with their
neighbors of degree~$2$.  Let~$N$ be the set of composed of all vertices of~$G_0-P$
that are adjacent to a vertex in~$Y''$ by an edge that does not belong to~$K$.
Note that $N$ is an independent set by the choice of~$Y'$.  Also
\refclaim{cl-chord4} yields that each vertex in~$P$ adjacent to a vertex in~$Y''$
has color~$1$ or~$2$.  Consider the graph~$G_0-Y''$ with the list assignment~$L$
such that 
\[
      L(v)=\begin{cases}
            \{\psi(v)\}&\quad\text{if~$v\in V(P)$,}\\
            \{1,2\}&\quad\text{if~$v\in (S\cap V(G_0-Y''))\cup N$,}\\
            \{1,2,3\}&\quad\text{otherwise.}
      \end{cases}
\]
Any $L$-coloring of~$G_0-Y''$ can be extended
to a $3$-coloring of~$C$ by first giving vertices in~$Y''$ color~$3$ and next
coloring~$C_z$ for each~$z\in T_p$; if $C_z$ contains a vertex of $Y''$,
we can extend the coloring so that the demand of~$C_z$ is satisfied.
It follows that in the resulting $3$-coloring of~$C$, the weight of satisfied demands is
at least~$\omega(Y'')\ge w(T_p)/48\ge \alpha_1 w(T)$, which contradicts the
assumption that $C$ is a counterexample.

Therefore, $G_0-Y''$ is not $L$-colorable, and thus it violates one of the assumptions of
Lemma~\ref{lemma-same}.  The assumption~(i) is clearly satisfied.  If
a vertex~$v\in S$ is adjacent to a vertex~$x\in N$ with a neighbor~$y\in Y''$,
then either
$yx$ is a chord of~$K$ or~$yxv$ is a weak $2$-chord of~$K$, and thus $yxv$ is
a subpath of the outer face of~$G_0$.  Suppose that the assumption~(ii) is
violated for a path~$v_1v_2v_3$. Then $v_1,v_3\in N$, $v_2\in S$, and the outer
face of~$G_0$ contains a subpath~$yv_1v_2v_3y'$ with~$y,y'\in Y''$.  However, this
contradicts the choice of~$Y''$, as $y$ and~$y'$ would then be consecutive in the
cycle~$y_1\dotso y_n$.  Finally, suppose that the assumption~(iii) is violated,
and thus the outer face of~$G_0$ contains a walk $yv_1v_2p_1p_2p_3v_3v_4y'$
(where $P=p_1p_2p_3$) with~$y,y'\in Y''$, $v_1,v_4\in N$ and~$v_2,v_3\in S$.  This
implies that $\{y,y'\}=\{y_1,y_n\}$, and so the choice of~$Y''$ implies that $n=1$ and $y=y'$.
By~\refclaim{cl-nosep}, the interior of the $8$-cycle $yv_1v_2p_1p_2p_3v_3v_4$ in $G$
contains no vertices, and hence $V(G_0)=V(P)\cup\{y,v_1,v_2,v_3,v_4\}$.  This implies that $G_0-y$ is
$L$-colorable, a contradiction.  We thus conclude the following.
\claim{cl-Tp}{We have $w(T_p)<48\alpha_1w(T)$.}

Let~$S_0=S\cap V(G_0)$ and $T_0=T_2\setminus T_p$.  From now on, we consider the
cog~$C_0=(G_0,P,S_0,T_0,w\restriction T_0)$.  Note that any $3$-coloring of~$C_0$
extends to a $3$-coloring of~$C$ (without necessarily satisfying any additional
demands).  Also, the outer face of~$G_0$ is bounded by a cycle~$K_0$.
\claim{cl-nogap1}{The graph~$G_0$ contains no path~$v_1v_2v_3$ with~$v_1,v_3\in
      S_0\cup T_0$ and~$v_2\not\in T_0$.}
      
Indeed, by~\refclaim{cl-gap1mod} such a path would be a subpath of~$K$ and $v_2$
would be incident with a chord or a weak $2$-chord, implying that $v_1$ or~$v_3$
belongs to~$V(G)\setminus V(G_0)$.

For~$t\in T_0$, let~$B_t$ be the set consisting of~$t$ and its two neighbors
in~$S_0$.  By~\refclaim{cl-nogap1}, if $t$ and $t'$ are two distinct vertices
in~$T_0$, then no vertex of~$G_0$ has neighbors both in~$B_t$ and $B_{t'}$.
Let~$G'_0$ be the graph obtained from~$G_0$ by, for each~$t\in T_0$, contracting
the edges between~$t$ and its neighbors in~$S_0$, and by removing all edges among
the neighbors of~$t$ in the resulting graph (since $G_0$ has girth at least~$5$, we
know by~\refclaim{cl-nosep} that there may be only one such edge, in case that $t$ has
degree two and is incident with a $5$-face).  Note that $G'_0$ is plane and
triangle-free, and by Corollary~\ref{cor-dist13}, there exists
a set~$T'_0\subseteq T_0$ such that $w(T'_0)\ge w(T_0)/16$ and no two vertices
of~$T'_0$ are joined by a path of length~$3$ in~$G'_0$.  Consequently, if $t,t'\in
T'_0$ are distinct, then $G_0$ contains no path of length~$3$ with one end
in~$B_t$ and the other end in~$B_{t'}$.

Let~$B=\bigcup_{t\in T'_0} B_t$ and let~$N$ be the set of vertices
in~$V(G_0)\setminus B$ that have a neighbor in~$B$.  By the previous paragraph,
$N$ induces a partial matching in~$G_0$ (with each edge of~$G_0[N]$ being
contained in the neighborhood of~$B_t$ for some $t\in T'_0$ of degree two, called
the \emph{origin} of the edge).  Furthermore, vertices of~$N$ have no neighbors
in~$S_0\setminus B$ by~\refclaim{cl-nogap1}, and thus $G_0[S_0\cup N]$ is
a partial matching with the same edges as $G_0[N]$. Observe also that,
by~\refclaim{cl-chord2}, \refclaim{cl-2chord} and the construction of~$G_0$,
the endvertices of~$P$ are not adjacent to vertices incident with an edge
of~$G_0[N]$.

Let~$p_1,\dotsc,p_k,s_1, \dotsc, s_{2|N|}$ be the vertices of~$P$ and of~$B\cap S_0$
in order around the outer face of~$G_0$.  Let~$p'_1,\dotsc,p'_k$ be new vertices,
and let~$G''_0$ be the graph obtained from~$G_0$ by adding the
cycle~$K'=p'_1\dotsc p'_ks_1\dotsc s_{2|N|}$ as its outer face as well as the
edges~$p_ip'_i$ for~$i\in\{1,\dotsc,k\}$.  Let~$G'_0$ be the graph obtained
from~$G''_0-(B\cap T_0)$ by removing all edges between~$B\cap S_0$
and~$V(G'_0)\setminus V(K')$ not incident with the vertices in~$N$.  Note that
$G'_0$ forms a casing for~$G_0-B$, $P$, and~$E(G_0[N])$; let~$\prec$ be the
corresponding ordering on the vertices incident with the edges of~$G_0[N]$.

Let~$w_N$ be the sum of the weights of the origins of the edges of~$G_0[N]$.
Let~$H$ be the bipartite graph with one part consisting of the vertices in~$N$
incident with the edges of~$G_0[N]$, and the other part of the vertices
in~$V(G_0)\setminus B$ that are adjacent to them in~$G_0$, and the edge set
consisting exactly of the edges of~$G_0$ between these two parts.  Let~$H'$ be the
graph obtained from~$H$ by, for each edge~$xy$ of~$G_0[N]$ with~$x\prec y$,
subdividing all edges of~$H$ incident with~$x$ once and then identifying~$x$ and~$y$
to a single vertex.  Note that $H'$ is plane and triangle-free, and thus by
Corollary~\ref{cor-dist13}, there exists a subset~$X$ of the edges of~$G_0[N]$
such that the corresponding vertices of~$H'$ are not joined by paths of length~$3$
and the set~$T_X$ of the origins of the edges in~$X$ satisfies $w(T_X)\ge w_N/16$.

Let~$T''_0$ be the set consisting of the vertices in~$T_X$ and of the vertices
of~$T'_0$ that are not origins of any edge of~$G_0[N]$.  Note that $w(T''_0)\ge
w(T'_0)/16\ge w(T_0)/256$.  Let~$B''=\bigcup_{t\in T''_0} B_t$ and let~$N''$ be
the set of vertices of~$V(G_0)\setminus B''$ that have a neighbor in~$B''$.  By
the construction of~$H'$ and the choice of~$X$, the following holds.
\claim{cl-iii}{If $x_1y_1$ and~$x_2y_2$ are distinct edges in~$G_0[N'']$ with
      $x_1\prec y_1$ and $x_2\prec y_2$, then $x_1$ and~$y_2$ have no common
      neighbors in~$G_0$, and $y_1$ and~$x_2$ have no common neighbors in~$G_0$.}

If $\abs{V(P)}\le 2$, then let~$T'''_0=T''_0$. Otherwise, if $P=p_1p_2p_3$, we
choose~$T'''_0\subseteq T''_0$ as follows.  For~$i\in \{1,3\}$, let~$O_i$ be the
set of edges~$xy\in E(G_0[N''])$ such that there exists a path~$p_iuvxy$
in~$G_0-B''$ with $u\in S_0\cup N''$; and let~$R_i$ denote the set of origins of
the edges in~$O_i\setminus O_{4-i}$.  By symmetry, we can assume that $w(R_1)\le
w(R_3)$.  We let~$T'''_0=T''_0\setminus R_1$, and note that $w(T'''_0)\ge
w(T''_0)/2\ge w(T_0)/512$.  Let~$B'''=\bigcup_{t\in T'''_0} B_t$.

Let~$c$ be a color in~$\{1,2\}$, different from~$\psi(p_2)$ when $\abs{V(P)}=3$.
Let~$L$ be the list assignment for~$G_0-B'''$ such that 
\[
      L(v)=\begin{cases}
            \{\psi(v)\}&\quad\text{if~$v\in V(P)$,}\\
            \{1,2\}&\quad\text{if~$v\in S_0\setminus B'''$,}\\
            \{1,2,3\}\setminus\{3-c\}&\quad\text{if $v$ is adjacent to a vertex in~$T_0'''$,}\\
            \{1,2,3\}\setminus\{c\}&\quad\text{if $v$ is adjacent to a vertex in~$S_0\cap B'''$,}\\
            \{1,2,3\}&\quad\text{otherwise.}
      \end{cases}
\]
Note that $G_0-B'''$ and the list assignment~$L$ satisfy the assumptions of
Lemma~\ref{lemma-dvokaw-strong} (the condition~(i') is obviously satisfied, the
condition~(ii) holds by the choice of~$T'_0$, the condition~(iii') holds
by~\refclaim{cl-iii}, and the condition~(iv') holds by the choice of~$T'''_0$ and
the color~$c$).  Hence, $G_0-B'''$ is $L$-colorable, and we can extend this
coloring to a $3$-coloring of~$C_0$ by giving vertices of~$T'''_0$ the color~$3-c$
and the vertices of~$B'''\cap S_0$ the color~$c$.  This satisfies all demands
in~$T'''_0$, whose total weight is at least~$w(T_0)/512$.  As this $3$-coloring
extends to~$C$, we have a contradiction unless $w(T_0)/512< \alpha_1w(T)$.

However, if $w(T_0)/512< \alpha_1w(T)$ then \refclaim{cl-T1} and~\refclaim{cl-Tp} yield that
\[
      w(T)=w(T_1)+w(T_p)+w(T_0)<(2+48+512)\alpha_1w(T)=w(T),
\]
which is a contradiction. This concludes the proof.
\end{proof}

We now generalize Lemma~\ref{lemma-polished} to triangle-free non-polished cogs (allowing
now only a path with two vertices to be precolored).

\begin{lemma}\label{lemma-full}
Let~$\alpha_0=\alpha_1/9$, where~$\alpha_1$ is the constant from
Lemma~\ref{lemma-polished} (i.e., $\alpha_0=1/5058$).  Let~$C=(G,P,S,T,w)$ be
a plane cog of girth at least~$4$, where $\abs{V(P)}\le 2$.  If either
$\abs{V(P)}\le 1$ or at least one vertex of~$P$ has no neighbor in~$S$, then every
$3$-coloring of~$P$ extends to a $3$-coloring of~$C$ satisfying
$\alpha_0$-fraction of the demands.
\end{lemma}
\begin{proof} Suppose for a contradiction that $C$ is a counterexample
      with~$\abs{V(G)}$ as small as possible, and let $\psi$ be a $3$-coloring of~$P$
      that does not extend to a $3$-coloring of~$C$ satisfying 
      $\alpha_0$-fraction of the demands.  Clearly, $G$ is connected and all vertices not
      belonging to~$S\cup T\cup V(P)$ have degree at least three.

Also, $G$ is $2$-connected: otherwise, let~$v$ be a cutvertex of~$G$, and
let~$C_1$ and~$C_2$ be the $v$-components of~$C$.  By the minimality of~$C$, the precoloring~$\psi$
extends to a $3$-coloring~$\varphi_1$ of~$C_1$ satisfying
$\alpha_0$-fraction of its demands.  Furthermore, the $3$-coloring of~$v$ by
color~$\varphi_1(v)$ extends to a $3$-coloring~$\varphi_2$ of~$C_2$ satisfying
$\alpha_0$-fraction of its demands.  The combination of~$\varphi_1$ and
$\varphi_2$ is a $3$-coloring of~$C$ satisfying $\alpha_0$-fraction of its
demands, which contradicts the assumption that $C$ is a counterexample.

Hence, the outer face of~$G$ is bounded by a cycle~$K$.
If $\abs{V(P)}\le 1$, then let~$S'=S$, otherwise let~$S'$ consist of~$S$ and a vertex of~$P$ that has no neighbor in~$S$.  
Suppose that $K$ has
a chord~$uv$, where~$u\in S'$.  Let~$C_1$ and~$C_2$ be the $uv$-components of~$C$.
Note that $u$ has no neighbor in~$S$, and thus $C_2$ satisfies the assumptions of
Lemma~\ref{lemma-full}.  Hence, we obtain a contradiction as in the previous
paragraph, and we conclude that $K$ has no chords incident with vertices in~$S'$.

By Theorem~\ref{thm-grotzsch}, it similarly follows that the open subset of the
plane contained inside any $(\le\!5)$-cycle in~$G$ is a face of~$G$.  Suppose that $G$
contains a $4$-face~$f=v_1v_2v_3v_4$.  If $f$ is the outer face, then we conclude that
$V(G)=\{v_1,v_2,v_3,v_4\}$ and it is easy to verify that every $3$-coloring
of~$P$ extends to a $3$-coloring of~$C$ satisfying $\alpha_0$-fraction of its
demands.  Hence, $f$ is not the outer face.  

Since $S'$ is an independent set, we can by symmetry assume that $v_1,v_3\not\in S'$.  
Furthermore, $G$ contains no path~$v_1xyv_3$ of length
three: otherwise, the face~$f$ would be contained in the interior of one of the
$5$-cycles~$v_1xyv_3v_2$ and~$v_1xyv_3v_4$, thereby contradicting our previous
conclusion that the interior of each $5$-cycle of~$G$ is a face.  Let~$C'$ be the cog
obtained from~$C$ by identifying~$v_1$ with~$v_3$ to a new vertex~$v$ (if
both~$v_1$ and~$v_3$ belong to~$T$, then $v$ has weight $w(v_1)+w(v_3)$ in~$C'$).
Note that $C'$ satisfies all the assumptions of Lemma~\ref{lemma-full}, and by the
minimality of~$C$, every $3$-coloring of~$P$ extends to a $3$-coloring of~$C'$
satisfying $\alpha_0$-fraction of its demands.  We can extend this $3$-coloring
to~$C$ by giving both~$v_1$ and~$v_3$ the color of~$v$.  Observe that the
resulting $3$-coloring satisfies $\alpha_0$-fraction of the demands of~$C$, unless
say $v_1\in V(P)$, $\psi(v_1)=3$ and $v_3\in T$.  Since $C$ is a counterexample,
the latter must be the case.

If $v_2,v_4\not\in S'$, we can identify~$v_2$ with~$v_4$ instead and obtain
a contradiction in the same way.  Hence, we can assume that $v_2\in S'$.  Since
$K$ has no chords incident with vertices in~$S'$, we conclude that $v_1v_2v_3$ is
a subpath of~$K$ and $v_2$ has degree two.  By the minimality of~$C$, there exists
a $3$-coloring~$\varphi$ of the subcog of~$C$ obtained by removing~$v_2$,
extending~$\psi\restriction (V(P)\setminus\{v_2\})$ and satisfying
$\alpha_0$-fraction of the demands.  If $v_2\in S$, then we can give~$v_2$ a color
in~$\{1,2\}\setminus\{\varphi(v_3)\}$, since $\psi(v_1)=3$.  If $v_2\in V(P)$,
then we can assume that $\varphi(v_3)\neq\psi(v_2)$, since $\psi(v_1)=3$,
$\psi(v_2)\in\{1,2\}$, and exchanging colors~$1$ and~$2$ in the coloring~$\varphi$
keeps the same weight of satisfied demands.  In either case, we obtain
a contradiction with the assumption that $C$ is a counterexample.  It follows that
$G$ has girth at least five.

By Theorem~\ref{thm-grotzsch}, there exists a $3$-coloring~$\psi_1$ of~$G$.
We write~$K=v_1v_2\dotsc v_k$, and note that there exists an assignment~$\psi_2$ of
colors in~$\{1,2,3\}$ to the vertices of~$K$ so that no two vertices at distance (in~$K$) exactly
two from each other have the same color.  Let~$T_1$ be a subset of~$T$ of maximum
weight that is monochromatic both in~$\psi_1$ and in~$\psi_2$; clearly, $w(T_1)\ge
w(T)/9$.  Since $T_1$ is monochromatic in~$\psi_1$, it is an independent set
in~$G$.  Since $K$ has no chords incident with vertices of~$S$, if $v_i\in S$ has a neighbor~$v_j\in T$, then
$j\in\{i-1,i+1\}$, with indices taken cyclically, and since $T_1$ is monochromatic
in~$\psi_2$, at most one such neighbor belongs to~$T_1$.  Hence, $G$ contains no
path~$u_1u_2u_3$ with~$u_2\in S$ and $u_1,u_3\in T_1$.

Therefore, $C'=(G,P,S,T_1,w\restriction T_1)$ is a polished plane cog of girth at
least~$5$, and by Lemma~\ref{lemma-polished}, every $3$-coloring of~$P$ extends to
a $3$-coloring~$\varphi$ of~$C'$ that satisfies $\alpha_1$-fraction of its
demands.  Note that $\varphi$ is also a $3$-coloring of~$C$, and since $w(T_1)\ge
w(T)/9$, it satisfies $(\alpha_1/9)$-fraction of the demands of~$C$.  This
contradicts the assumption that $C$ is a counterexample.
\end{proof}

The result on request graphs with only non-equality requests all at a single vertex
now readily follows.

\begin{proof}[Proof of Corollary~\ref{cor-single}] Let~$v$ be a common neighbor of vertices of
      $R_{\neq}$, and let~$T$ be the set of neighbors of vertices of~$R_{\neq}$ not equal
      to~$v$. For~$t\in T$, let us define $w'(t)=\sum_{r\in
            R_{\neq},tr\in E(G)} w(r)$.  Let~$S$ be the set of neighbors of~$v$ not belonging
            to~$R_{\neq}$.  Let~$C=(G-(R_{\neq}\cup\{v\}), \varnothing, S,T,w')$, and note that
            $C$ is a plane cog of girth at least~$4$.  By Lemma~\ref{lemma-full},
            there exists a $3$-coloring of~$C$ satisfying $\alpha_0$-fraction
            of its demands.  By giving~$v$ the color~$3$ and coloring vertices
            of~$R_{\neq}$ by colors different from the colors of their neighbors, we
            obtain a $3$-coloring of~$G$ that satisfies $\alpha_0$-fraction of
            its requests, as required.
\end{proof}

\end{document}